\documentclass[11pt]{amsart}
\usepackage[letterpaper, margin=1in]{geometry}

\usepackage{graphicx}
\usepackage{amsthm,amsmath,amssymb}
\usepackage{todonotes}
\usepackage{geometry}
\usepackage{xspace}
\usepackage{hyperref}
\usepackage{algorithm}
\usepackage[noend]{algpseudocode}
\usepackage{thmtools}
\usepackage{thm-restate}
\usepackage[T1]{fontenc}
\usepackage{xcolor}

\newtheorem{theorem}{Theorem}[section]

\newtheorem{lemma}[theorem]{Lemma}
\newtheorem{proposition}[theorem]{Proposition}
\newtheorem{observation}[theorem]{Observation}
\newtheorem{claim}{Claim}[theorem]

\newtheorem{conjecture}{Conjecture}[section]
\newtheorem{problem}[conjecture]{Problem}

\newcommand{\cal}{\mathcal}
\newcommand{\tw}{tw\xspace}

\newcommand{\atw}{tw_{\alpha}\xspace}

\newcommand{\lpopt}{LP_{opt}}

\newcounter{tbox}

\title{(Treewidth, Clique)-Boundedness and Poly-logarithmic Tree-Independence}

\author{Maria Chudnovsky$^{\dagger}$}
\author{Ajaykrishnan E S$^{\ddagger}$}
\author{Daniel Lokshtanov$^{\ddagger}$}

\thanks{$^{\ddagger}$ Department of Computer Science, University of California Santa Barbara, Santa Barbara, CA, USA}
\thanks{$^{\dagger}$Princeton University, Princeton, NJ, USA.  Supported by NSF-EPSRC Grant DMS-2348219 and by AFOSR grant FA9550-22-1-0083.}

\date{\today}

\begin{document}

\pagenumbering{gobble}

\begin{abstract}
An {\em independent set} in a graph $G$ is a set of pairwise non-adjacent vertices.
A {\em tree decomposition} of  $G$ is a pair $(T, \chi)$ where $T$ is a tree and $\chi : V(T) \rightarrow 2^{V(G)}$ is a function satisfying the following two axioms: for every edge $uv \in E(G)$ there is an $x \in V(T)$ such that $\{u,v\} \subseteq \chi(x)$, and for every vertex $u \in V(G)$ the set $\{x \in V(T) ~|~ u \in \chi(x)\}$ induces a non-empty and connected subtree of $T$.
The sets $\chi(x)$ for $x \in V(T)$ are called the {\em bags} of the tree decomposition.  The {\em tree-independence} number of $G$ is the minimum taken over all tree decompositions of $G$ of the maximum size of an independent set of the graph induced by a bag of the tree decomposition.

The study of graph classes with bounded  tree-independence number has attracted much attention in recent years, in part due to its important algorithmic implications. A conjecture
of Dallard, Milani\v{c} and \v{S}torgel, connecting tree-independence number to the classical notion of treewidth, was one of the motivating problems in the area.
This conjecture was recently disproved, but here we prove a slight variant of
it, that retains much of the algorithmic significance. As part of the proof we introduce the notion of {\em independence-containers}, which can be viewed as a generalization of the set of all maximal cliques of a graph, and is of independent interest. 
\end{abstract}

\maketitle

\newpage
\pagenumbering{arabic}

\section{Introduction}\label{sec:intro}
    
%%%%%%%%%%%%%%%%%%%%%%%%%%%%%%%%%%%%%%%%% SODA INTRO

Tree decompositions are a key tool in structural graph theory and graph algorithms. For a graph $G = (V(G),E(G))$, a \emph{tree decomposition} $(T, \chi)$ of $G$ consists of a tree $T$ and a map $\chi: V(T) \to 2^{V(G)}$ with the following properties: {\em (i)} For every $v_1v_2 \in E(G)$, there exists $t \in V(T)$ with $\{v_1, v_2\} \subseteq \chi(t)$, {\em (ii)} For every $v \in V(G)$, the subgraph of $T$ induced by $\{t \in V(T) \mid v \in \chi(t)\}$ is non-empty and connected.
The \emph{width} of a tree decomposition $(T, \chi)$
% , denoted by $\w(T, \chi)$, 
is $\max_{t \in V(T)} |\chi(t)|-1$. The \emph{treewidth} of $G$, denoted by $\tw(G)$, is the minimum width of a tree decomposition of $G$. Bounded treewidth is a fundamental graph property from both a structural \cite{bodlaender1998partial,RS-GMXVI} and an algorithmic \cite{cygan2015parameterized} perspective, for a recent survey see~\cite{korhonen2024computing}.

For a graph $G$ and $X \subseteq V(G)$ we denote by $G[X]$ the {\em subgraph of $G$ induced by $X$}, that is the graph with vertex set $X$ in which two vertices are adjacent if and only if they are adjacent in $G$. We denote by $G - X$ the graph $G[V(G) \setminus X]$.  A class of graphs is {\em hereditary} if it is closed under taking induced subgraphs.
Tree decompositions have traditionally been studied in the context of graph minors, but in recent years their behavior in  hereditary graph classes has come into the spotlight. Additionally, more nuanced ways of measuring the complexity of a tree decomposition than just its width have been considered
(see e.g.~\cite{beyondTreewidthSurvey,Yolov}).

{\em Tree-independence number}, which we will define next,  is one such complexity measure, which has recently received substantial attention. An {\em independent set}  in a graph $G$ is a set of pairwise non-adjacent vertices of $G$. The {\em independence} number 
$\alpha(G)$ of $G$ is the maximum size of an independent set in $G$. For a vertex subset $X\subseteq V(G)$, we use $\alpha(X)$ to denote $\alpha(G[X])$.
%
%A {\em clique} is a set of pairwise adjacent vertices, and the {\em clique number} $\omega(G)$ of a graph $G$ is the maximum size of a clique in $G$.
%
The {\em independence number} of a tree decomposition $(T, \chi)$ of $G$ is $\max_{t \in V(T)} \alpha(\chi(t))$. The {\em tree-independence number} of $G$, denoted $\atw(G)$, is the minimum independence number of a tree decomposition of $G$.
Tree-independence was first defined by Yolov~\cite{Yolov}, and independently re-discovered by Dallard et al.~\cite{dms2}, who initiated the study of tree-independence number in the context of graphs whose treewidth is bounded by a function of their clique number. 
Both Yolov~\cite{Yolov} and  Dallard et al.~\cite{dms2} observed that the {\sc Independent Set} problem (given as input $G$, compute $\alpha(G)$), as well as a few other problems which are NP-hard on general graphs, can be solved in polynomial time in graphs with
bounded tree-independence number (even if a tree decomposition of constant independence number is not explicitly given).  Lima et al.~\cite{lima2024tree} added a number of even more general problems to this  list. Tree-independence number also has connections to coarse geometry \cite{CoarseGeoKwon, Agelos}.

Motivated by these considerations, Dallard et al.~\cite{dms3} initiated the systematic study of graph classes with {\em bounded tree-independence number}, namely graph classes ${\cal C}$ for which there exists a universal constant $c$ such that every graph $G \in {\cal C}$ satisfies $\atw(G) \leq c$. 
A {\em clique} in a graph $G$ is a set of pairwise adjacent vertices.
For a graph $G$ we denote by $\omega(G)$ the largest size of a clique in $G$.
A hereditary class of graphs $\cal C$ is said to be
$(\tw, \omega)$-bounded if there exists a function $f$ such that
$\tw(G) \leq  f(\omega(G))$ for every  $G \in \cal C$.
Such a function $f$ is called a 
{\em $(\tw,\omega)$-bounding function} for the class $\mathcal{C}$.
We say that $\mathcal{C}$ is {\em polynomially} $(\tw,\omega)$-bounded if
there is a polynomial $(\tw,\omega)$-bounding  function for $\mathcal{C}$.
By Ramsey's theorem \cite{Ramsey}, graph classes of bounded tree-independence number are polynomially $(\tw,\omega )$-bounded. It was conjectured in \cite{dms3} that a converse implication holds as well.  This statement became known as
``Dallard-Milani\v{c}-{\v{S}}torgel Conjecture''.
\begin{conjecture}\label{Milanic}
  Let $\cal C$ be a hereditary graph class. Then, if there exists a function $f$ such that for every $G \in {\cal C}$ it holds that $\tw(G) \leq f(\omega(G))$, then ${\cal C}$ has bounded tree-independence number.
 \end{conjecture} 
Conjecture~\ref{Milanic} received a significant amount of attention in the graph theoretic community. 
It was recently disproved by Chudnovsky and Trotignon~\cite{CT} (even if $f$ is a polynomial function). 

For a function $f: \mathbb{N} \rightarrow \mathbb{N}$, we say that $f$ is {\em poly-logarithmic}
if there exist positive integers $c,d$ such that $f(n) \leq c \log^d n$ for every $n \geq 3$.
We say that a hereditary graph class $\mathcal{C}$ has 
{\em poly-logarithmic tree-independence number}
if there exists a poly-logarithmic function $f$ such that every $G \in \mathcal{C}$ satisfies
$\atw(G) \leq f(|V(G)|)$. 
A function $f:\mathbb{N} \rightarrow \mathbb{N}$ is {\em quasi-polynomial} if 
there exist positive integers $c,d$ such that $f(n) \leq c n^{\log^d n}$ for every $n$.
Several of the algorithms mentioned above~\cite{dms2,lima2024tree,Yolov}, most prominently for {\sc Independent Set}, run in quasi-polynomial time  
%(that is time $cn^{\log^d n}$  for an $n$-vertex graph, where $c$  and $d$ are constants that work for all the graphs in question) 
when the input is restricted to a graph class with poly-logarithmic  tree-independence number.
A quasi-polynomial-time algorithm for a problem, while not quite as efficient as a polynomial-time algorithm, shows that the problem is not NP-hard unless every problem in NP can be solved in quasi-polynomial time, a complexity theoretic collapse that is viewed almost as unlikely as P=NP. We refer the reader to \cite{TI2,sparseinduced} for a more detailed discussion of the algorithmic applications of poly-logarithmic bounds on the tree-independence number. 

Several graph classes have recently been shown to have
poly-logarithmic tree-independence number (\cite{TI2,ChudnovskyGHLS25}), yielding the first quasi-polynomial-time algorithms for {\sc Independent Set} on these graph classes.  We remark that the existence of a {\em polynomial} time algorithm for {\sc Independent Set} for the graph class studied in \cite{ChudnovskyGHLS25} remains a prominent open problem in the field (see~\cite{ChudnovskyGHLS25} and references within). 
By now a number of graph classes with poly-logarithmic (but not bounded by a constant) tree-independence number have been identified 
\cite{mainconj, deathstar, awesomegraphparameters}. 
This motivates  a systematic study of graph classes with poly-logarithmic tree-independence number. Our main result is that, in order to determine if a graph class $\mathcal{C}$ has poly-logarithmic tree-independence number, it is enough to study 
the dependence of treewidth on clique number in graphs in this class.
We prove the following:

\begin{restatable}{theorem}{MainThmTreeAlpha}\label{thm:main}
    Let $\mathcal{C}$ be a hereditary graph class. The following are equivalent:
    \begin{enumerate}\setlength\itemsep{-.7pt}
        \item[(i)]\label{itm:atwbounded} There exists a positive integer $c_1$ such that for every $G \in \mathcal{C}$ on at least $3$ vertices we have $\atw(G) \leq (\log |V(G)|)^{c_1}$.
        \item[(ii)]\label{itm:atwcliquebounded} There exists a positive integer $c_2$ such that for every $G \in \mathcal{C}$ on at least $3$ vertices we have $\atw(G) \leq (\omega(G) \log |V(G)|)^{c_2}$.
        \item[(iii)]\label{itm:twcliquebounded} There exists a positive integer $c_3$ such that for every $G \in \mathcal{C}$ on at least $3$ vertices we have $\tw(G) \leq (\omega(G) \log |V(G)|)^{c_3}$.
    \end{enumerate}
\end{restatable}

A step in the proof of Theorem~\ref{thm:main} is the following result, which we believe to be of independent interest. For two non-adjacent vertices $a,b$
we say that $X \subseteq V(G)$ {\em separates} $a$ and $b$  if
$\{a,b\} \cap X = \emptyset$, and no component
of $G - X$ contains both $a$ and $b$. We prove:

\begin{restatable}{theorem}{MainThmSeparator}\label{thm:ab_separator_in_class}
    For every positive integer $c$ there exists an integer $d=d(c)$ with the following property. If $\mathcal C$ is a hereditary graph class such that for every $G \in \mathcal C$ 
    %on at least $3$ vertices 
    and for every two non-adjacent vertices $u,v \in V(G)$, there exists a set $X \subseteq V(G)$ disjoint from $\{u,v\}$ with $|X| \leq (\omega(G) \log |V(G)|)^c$ that separates $u$ from $v$, then for every $G \in \mathcal C$
    %on at least $3$ vertices 
    and for every two non-adjacent vertices $u,v \in V(G)$, there exists a set $X \subseteq V(G)$ disjoint from $\{u,v\}$, with $\alpha(X) \leq \log^d (|V(G)|)$, that separates $u$ from $v$.
\end{restatable}

% \begin{theorem}\label{thm:main}
% Let $\mathcal{C}$ be a hereditary graph class. The following are equivalent:
% \begin{enumerate}\setlength\itemsep{-.7pt}
%     \item[(i)]\label{itm:atwbounded} $\exists c_1 > 0$ such that for every $G \in \mathcal{C}$ on at least $3$ vertices we have $\atw(G) \leq (\log |V(G)|)^{c_1}$
    
%     %There exists a constant $c_1$ such that, for every $G \in \mathcal{C}$, we have $\atw(G) \leq (\log |V(G)|)^{c_1}$
%     \item[(ii)]\label{itm:atwcliquebounded} $\exists c_2 > 0$ such that for every $G \in \mathcal{C}$ on at least $3$ vertices we have $\atw(G) \leq (\omega(G) \log |V(G)|)^{c_2}$
    
%     %There exists a constant $c_2$ such that, for every $G \in \mathcal{C}$, we have $\atw(G) \leq (\omega(G) \log |V(G)|)^{c_2}$
%     \item[(iii)]\label{itm:twcliquebounded} $\exists c_3 > 0$ such that for every $G \in \mathcal{C}$ on at least 3 vertices we have $\tw(G) \leq (\omega(G) \log |V(G)|)^{c_3}$
    
%     %There exists a constant $c_3$ such that, for every $G \in \mathcal{C}$, we have $\tw(G) \leq (\omega(G) \log |V(G)|)^{c_3}$
% \end{enumerate}
% \end{theorem}

We note that in Theorem~\ref{thm:main}, the implication {\em (i) $\rightarrow$ (ii)} is trivial, while  {\em (ii) $\rightarrow$ (iii)} follows directly from Erd{\H{o}}s and Hajnal's results~\cite{erdos1989ramsey} regarding the Ramsey numbers in graphs excluding a complete bipartite graph (see Section~\ref{sec:proofOfTheorems}). The majority of this paper deals with proving the implication {\em (iii) $\rightarrow$ (i)}, which can be viewed as a slightly weakened version of Conjecture~\ref{Milanic}.

%\subsection{Methods}
The starting point for our work is a recent approximation algorithm for computing the fractional hypertreewidth of an input hypergraph, due to Korchemna et al.~\cite{KorchemnaL0S024}. Korchemna et al.~\cite{KorchemnaL0S024} developed a toolbox to deal with separation problems in a graph $G$, when given a family ${\cal F}$ of cliques in $G$, and the task is to find separators in $G$ that are covered by few sets in ${\cal F}$.
We start by generalizing their arguments to the case when the sets in ${\cal F}$ have bounded independence number, as opposed to being cliques. %The main difference in the arguments is in the proof of Claim~\ref{claim:S_r_has_small_cover}, as well as some minor technical differences in Section~\ref{sec:balanced_separator}, otherwise the arguments in Sections~\ref{sec:ab_seps},~\ref{sec:ab_sampling} and~\ref{sec:balanced_separator} very closely follow corresponding proofs of Korchemna et al~\cite{KorchemnaL0S024}. 

Our second contribution is to initiate the study of {\em independence-containers}, a combinatorial object which naturally shows up in our arguments, and which we believe is worth further investigation in its own right. Let $G$ be a graph. Given two sets $F,H\subseteq V(G)$, we say that $F$ {\em covers} $H$ if $H \subseteq F$. A family $\mathcal{F}$ of vertex subsets of $G$ \emph{covers} a vertex set $H$ if there exists an $F\in\mathcal{F}$ that covers $H$. Let $a$ and $b$ be positive integers. Then 
$\mathcal{F}$ is a $(b,a)$-{\em container family} of $G$ if 
$\alpha(F)\leq a$ for every $F\in\mathcal{F}$ and 
$\mathcal{F}$ covers every vertex set $H$ in $G$ such that $\alpha(H) \leq b$.
Independence-containers are tangentially related to, but should not be confused with, containers as used in the hypergraph container method~\cite{saxton2015hypergraph,balogh2018method}.
%\todo[inline]{Both in hypergraph containers, and in my work with Marcin et al, the word "container" referes to the  elements of $\mathcal{F}$, and not to the whole collection $\mathcal{F}$. I think it would be nice to keep that terminology consistent. DL:TODO change container to contaienr family everywhere. DONE, i think.}

Independence-containers can be seen as a generalization of the family of all {\em maximal cliques} in a graph $G$. (A clique $C$ in a graph $G$ is {\em maximal} if no proper superset of $V(C)$ induces a clique in $G$.) It is easy to see that the unique inclusion minimal (1,1)-container family ${\cal F}$ in $G$ is precisely the set of all maximal cliques in $G$. The notion of $(1,a)$-container families generalizes the notion of maximal cliques by allowing one to cover the cliques of $G$ using sets of independence number at most $a$. On the other hand $(b,a)$-container families require that all sets $H$ with independence number at most $b$ are covered, not only cliques. 
It is known that for every positive integer $k$ there exists an integer $t$ such that if a graph $G$ excludes the $\overline{kK_2}$ as an induced subgraph (see Section~\ref{sec:prelim} for a definition of the graph $\overline{kK_\omega}$) then $G$ has at most $n^{t}$ maximal cliques~\cite{BalasY89}. On the other hand, the $\overline{kK_2}$ has $2^k$ maximal cliques, and a $(1,2)$-container family ${\cal F}$ of size $1$, namely ${\cal F} = \{V(\overline{kK_2})\}$. 
Thus, by relaxing which sets we can use to cover the cliques we (sometimes) can cover all the cliques in $G$ using many fewer sets. This raises the question of which classes of graphs have ``efficient'' independence-containers in the sense that we would like $|{\cal F}|$ and the independence number $a$ of the sets in ${\cal F}$ to be as small as possible as a function of $n$ and the independence number $b$ of the sets to be covered. We characterize the hereditary classes of graphs that have independence-containers with quasi-polynomial $|{\cal F}|$ and poly-logarithmic $a$: it is precisely the classes which exclude $\overline{kK_k}$ for at least one $k$. Specifically, we prove the following.

\begin{restatable}{theorem}{MainThmContainer}\label{thm:containerIff}
    Let ${\cal C}$ be a hereditary class of graphs and $b : \mathbb{N} \rightarrow \mathbb{N}$ be a 
    poly-logarithmic function.
    Then the following are equivalent:
    \begin{enumerate}\setlength\itemsep{-.7pt}
        \item[(i)] There exist  a poly-logarithmic function $a$ and a quasi-polynomial function $f$  such that for every $n$-vertex graph $G \in {\cal C}$, $G$ has a $(b(n), a(n))$-container family ${\cal F}$ with $|\mathcal{F}| \leq f(n)$. 
        \item[(ii)] There exists a positive integer $k$ such that  $\overline{kK_k} \notin {\cal C}$.
    \end{enumerate}
\end{restatable}

% \begin{theorem}\label{thm:containerIff}
% Let ${\cal C}$ be a hereditary class of graphs, and $b : \mathbb{N} \rightarrow \mathbb{N}$ be a function such that $b(n) = O((\log n)^{O(1)})$. Then the following are equivalent:
% \begin{enumerate}\setlength\itemsep{-.7pt}
%     \item[(i)] There exists a function $a : \mathbb{N} \rightarrow \mathbb{N}$ such that $a(n) = O((\log n)^{O(1)})$ such that for every $n$-vertex graph $G \in {\cal C}$, $G$ has a $(b(n), a(n))$-container family ${\cal F}$ of quasi-polynomial size. 
%     \item[(ii)] There exists a positive integer $k$ such that  $\overline{kK_k} \notin {\cal C}$.
% \end{enumerate}
% \end{theorem}

We leave it as an open problem whether the bounds of Theorem~\ref{thm:containerIff} can be tightened 
when the function $b$ is upper bounded by a constant independent of $n$. In particular it would be interesting to see whether it is possible in this case to improve the upper bound on $a$ to a constant independent of $n$, and the upper bound on the size of $|{\cal F}|$ to a polynomial in $n$.
That said, such an improvement would not have substantial implications for the bounds that we achieve for Theorem~\ref{thm:main}. %\todo{do we need this comment? It makes the open problem less appealing, but it explains why we didnt bother with it in this paper. {\em I would keep it. The open problem is natural and nice enough}}

A class ${\cal C}$ of graphs is said to have {\em polynomially many maximal cliques} if there exists a constant $c$ such that every $n$-vertex graph $G$ in ${\cal C}$ has at most $n^c$ different maximal cliques. A version of Theorem~\ref{thm:ab_separator_in_class} 
was proved (implicitly) in \cite{ChudnovskyGHLS25} for graph classes that have polynomially many maximal cliques (it  follows immediately from the proof of Theorem 3.1 in~\cite{ChudnovskyGHLS25}).  The proof of  Theorem~\ref{thm:ab_separator_in_class} 
uses a $(1, a)$-container family ${\cal F}$ instead of the set of all maximal cliques and applies our previously described results  for separators covered by a family ${\cal F}$ of sets with bounded independence number. The existence of a  $(1, a)$-container family ${\cal F}$ with the desired parameters follows  from Theorem~\ref{thm:containerIff} because every hereditary class ${\cal C}$ that satisfies the hypothesis of Theorem~\ref{thm:ab_separator_in_class} excludes the complete bipartite graph  $K_{2,k}$ for some $k$, and hence also excludes $\overline{2K_k}$.

% \begin{theorem}\label{thm:ab_separator_in_class}
% For every positive integer $c$ there exists an integer $d=d(c)$ with the following property. If $\mathcal C$ is a hereditary graph class such that for every $G \in \mathcal C$ on at least $3$ vertices and for every two non-adjacent vertices $u,v \in V(G)$, there exists a set $X \subseteq V(G)$ disjoint from $\{u,v\}$ with $|X| \leq (\omega(G) \log |V(G)|)^c$ that separates $u$ from $v$,
% %
% then for every $G \in \mathcal C$ on at least $3$ vertices and for every two non-adjacent vertices $u,v \in V(G)$, there exists a set $X \subseteq V(G)$ disjoint from $\{u,v\}$, with $\alpha(X) \leq \log^d (|V(G)|)$, that separates $u$ from $v$.
% \end{theorem}

The ``only'' difference between Theorem~\ref{thm:ab_separator_in_class} and the implication {\em (iii) $\rightarrow$ (i)} of Theorem~\ref{thm:main} is that Theorem~\ref{thm:ab_separator_in_class} deals with $u$-$v$ separators, while Theorem~\ref{thm:main} deals with balanced separators.
We take inspiration from the {\em proof} of Theorem~\ref{thm:ab_separator_in_class}, and proceed as follows. 
First we observe that every hereditary graph class ${\cal C}$ that satisfies {\em (iii)} of Theorem~\ref{thm:main} excludes $\overline{2K_t}$ for some $t$, and therefore also satisfies the assumption of Theorem~\ref{thm:containerIff}. Let ${\cal F}$ be the $(1, a)$-container family with quasi-polynomial size and poly-logarithmic bound $a$ on the independence number of the sets in ${\cal F}$ obtained by applying Theorem~\ref{thm:containerIff}. 

We aim to prove the implication {\em (iii) $\rightarrow$ (i)}; to that end we prove the contrapositive. We start with a graph $G \in {\cal C}$ with ``too large'' tree-independence number and prove that $G$ then  contains an induced subgraph $G'$ with $\omega(G')$ very small, and whose treewidth is large. 
%This contradicts the premise {\em (iii)} that every graph in ${\cal C}$ with small clique number also has small treewidth. 
%
Suppose now that $G \in {\cal C}$ has large tree-independence number. By the argument from the $4$-approximation algorithm for treewidth of Robertson and Seymour~\cite{RobertsonS95b} applied to $\atw$ (see Lemma~\ref{lem:tree_alpha_separator}), $G$ contains a large independent set $I$ such that for every set $S$ with $\alpha(S) \leq \frac{|I|}{40}$, there exists a component $C$ of $G - S$ such that $|I \cap C| \geq \frac{95|I|}{100}$. 
%\todo{confirm with Ajay about numbers\vspace{10pt} \emph{Done}}

We consider a variation of the integer linear program of Korchemna et al.~\cite{KorchemnaL0S024} for finding a balanced separator $S$ for $I$ in $G$ which is covered by few sets from ${\cal F}$. 
If the optimum value of the linear programming relaxation of this ILP is small, then the rounding procedure of Korchemna et al. (adapted to ${\cal F}$ with small independence number) yields that $I$ has a balanced separator $S$ covered by few sets in ${\cal F}$. 
Since each set in ${\cal F}$ has very small independence number, this implies that the independence number of $S$ is also small compared to $|I|$, contradicting the separation properties of $I$ stated above. 
Thus we may assume that the optimum value of this LP relaxation is large;   say at least $f$ where $f \geq \frac{|I|}{\log^cn}$  for
an appropriately chosen integer $c$ (we may assume that $n \geq 3$). 

We now consider a dual optimal solution to this LP, and re-interpret it as a probability distribution ${\cal D}$ on induced paths between vertices in $I$. We show that this distribution (essentially) satisfies the following two properties: $(i)$ for every set $F \in {\cal F}$ a path $P$ sampled from ${\cal D}$ intersects $F$ with probability at most $1/f$, and $(ii)$ for every partition of $I$ into $(I_1, I_2)$ with $\max\{|I_1|, |I_2|\} \leq \frac{2|I|}{3}$, the probability that one end of $P$ is in $I_1$ and the other in $I_2$ is at least $\frac{1}{3}$.

We sample about $f \cdot \log |{\cal F}|$ paths $P_1, \ldots, P_\ell$ from this distribution, and set $G' = G[I \cup \bigcup_{i=1}^\ell V(P_i)]$. With high probability $I$ has no small balanced separator in $G'$ (so $G'$ has large treewidth), and no set in ${\cal F}$ intersects more than $\frac{6\ell}{f}$ paths in the sample, which is upper bounded by a poly-logarithmic function of $n$ because ${\cal F}$ has quasi-polynomial size. 
However, this implies that $\omega(G')$ is small! Indeed, consider a clique $C$ in $G'$. We have that $C \subseteq F$ for some $F \in {\cal F}$, since ${\cal F}$ is a $(1, a)$-container family.
The number of  paths $P_i$ in the sample that intersect $F$ (and therefore $C$) is upper-bounded by a poly-logarithmic function of $n$, and each such path intersects $C$ in at most two vertices, since the path is induced. Finally, $|C \cap I| \leq 1$ since $C$ is a clique and $I$ is an independent set, and thus we obtain 
a poly-logarithmic bound on  $|C|$. Thus we found an induced subgraph $G'$ of $G$ with large treewidth and small $\omega$ completing the proof of Theorem~\ref{thm:main}.
\\
\\
\noindent {\bf Overview of the paper.}
In Section~\ref{sec:prelim} we set up definitions and notation. 
In Section~\ref{sec:containers} we prove Theorem~\ref{thm:containerIff}. We do not need the full power of Theorem~\ref{thm:containerIff} in the rest of the paper, so a reader only interested in the proof of Theorems~\ref{thm:main} and~\ref{thm:ab_separator_in_class} can read Lemma~\ref{lem:containers1} and proceed to the next section. 
In Sections~\ref{sec:ab_seps}, \ref{sec:ab_sampling} and~\ref{sec:balanced_separator} we prove the generalized versions of the results of Korchemna et al.~\cite{KorchemnaL0S024} in the setting where ${\cal F}$ is a family of sets with bounded independence number rather than cliques. 
The results proved in Sections~\ref{sec:ab_seps} and~\ref{sec:ab_sampling} are used for the proof of Theorem~\ref{thm:ab_separator_in_class}, while the results of Sections~\ref{sec:ab_seps} and~\ref{sec:balanced_separator} are used as tools for the proof of Theorem~\ref{thm:main}.
Section~\ref{sec:dual_of_balanced_separator} contains the sampling argument which is at the core of the proof of Theorem~\ref{thm:main}. In Section~\ref{sec:proofOfTheorems} we combine the tools  developed so far and prove Theorems~\ref{thm:main} and~\ref{thm:ab_separator_in_class}.
We conclude with some final remarks and open problems in Section~\ref{sec:conclusion}.

\section{Preliminaries}\label{sec:prelim}
    %\todo[inline]{$K_{t,t}$, $t$ is used somewhere.w}
We denote by $n$ the number of vertices of the considered graph $G$. For two set families ${\cal F}_1$ and ${\cal F}_2$ we define ${\cal F}_1 \otimes {\cal F}_2 = \{S_1 \cup S_2 ~:~ S_1 \in {\cal F}_1, S_2 \in {\cal F}_2\}$. 
We use $[k]$ to denote the set $\{1, 2, \dots, k\}$ for a positive integer $k$, and $[x]$ to denote $[\left\lceil x \right\rceil]$ for a positive real number $x$. 
Unless the base is specified, logarithms are in base $2$.

A {\em clique} in a graph $G$ is a set of pairwise adjacent vertices. 
We use $\omega(G)$ to denote the size of the largest clique in $G$.
For a positive integer $t$, $K_t$ denotes a clique on $t$ vertices.
An \emph{independent set} in a graph $G$ is a set of pairwise non-adjacent vertices.
We use $\alpha(G)$ to denote the size of the largest independent set in $G$.
With slight abuse of notation, $\alpha(S)$ is used to denote $\alpha(G[S])$, for a vertex subset $S$.
For positive integers $t_1$ and $t_2$, the {\em complete bipartite graph} $K_{t_1, t_2}$ is
the graph with vertex set $A \cup B$, where $A$ and $B$ are disjoint sets of sizes
$t_1$ and $t_2$ respectively, and in which every vertex in $A$ is adjacent to every
vertex in $B$ and there are no other edges.
A \emph{walk} in a graph is a sequence of vertices in which each vertex is adjacent to the next, while a \emph{path} is a walk in which every vertex is distinct. An \emph{$A$--$B$ walk} (respectively, \emph{$A$--$B$ path}) is a walk (respectively, path) in which the first vertex belongs to $A$ and the last to $B$.
For two disjoint vertex subsets $A$ and $B$ in $G$, $A$ and $B$ are {\em anti-complete} if there is no edge in $G$ with one endpoint in $A$ and the other in $B$.
The {\em complement} of a graph $G$ is denoted by $\overline{G}$ and is the graph with vertex set $V(G)$ satisfying $uv\in E(\overline{G})$ if and only if $uv\notin E(G)$ for every $u,v\in V(G)$.
For a graph $G$ and positive integer $k$ the graph $kG$ is the graph obtained by taking $k$ disjoint copies of $G$ and making the copies anti-complete to each other. Note that the graph $\overline{2K_t}$ is the complete bipartite graph $K_{t,t}$. 
% \todo{I think we never need it for $a \neq b$ define $K_{a,b}$}

When $Z$ is a vertex subset of $G$, $G-Z$ denotes the graph induced by $V(G)\setminus Z$.
We use $N[Z]$ to denote $\{u\in V(G)\ |\ u\in Z, \text{ or } \exists v\in Z \text{ satisfying } u\in N(v)\,\}$. 
When clear from context, we use $H\cap Z$ to mean $V(H)\cap Z$ and $H\setminus Z$ to mean $V(H)\setminus Z$, where $H$ is a subgraph of a graph $G$.
For a positive real number $\phi$, a \emph{$(Z,\phi)$-balanced separator} is a vertex subset $S$ such that every connected component of $G-S$ contains at most $\phi \cdot |Z|$ vertices from $Z$.
If $A$, $B$ are vertex subsets, then an \emph{$A$--$B$ separator} is a vertex subset $S$ such that every $A$--$B$ path in $G$ contains at least one vertex from $S$.
Let $\mathcal{F}$ be a family of vertex subsets of $G$. A \emph{fractional $(A,B)$-separator} using $\mathcal{F}$ in $G$ is an assignment $\{x_F\}_{F\in\mathcal{F}}$ of non-negative real numbers to elements of $\mathcal{F}$ such that $\sum_{\substack{F\in\mathcal{F}\\ F\cap P\neq\emptyset}}x_F \geq 1$ for every $A$--$B$ path $P$. 
A \emph{fractional cover} of $Z$ using $\mathcal{F}$ is an assignment $\{x_F\}_{F \in \mathcal{F}}$ of non-negative real numbers to elements of $\mathcal{F}$ such that $\sum_{\substack{F \in \mathcal{F} \\ v \in F}} x_F \geq 1$ for every $v \in Z$. 
The \emph{fractional cover number} of $Z$ with respect to $\mathcal{F}$ denoted $\mathrm{fcov}_{\mathcal{F}}(Z)$ is the minimum of $\sum_{F\in\mathcal{F}}x_F$ over all fractional covers $\{x_F\}_{F \in \mathcal{F}}$ of $Z$.
%
% A subset $\mathcal{F'}$ of $\mathcal{F}$ is a \emph{cover} of $Z$ using $\mathcal{F}$ if every vertex of $Z$ is contained in at least one element of $\mathcal{F}'$.
%
The \emph{cover number} of $Z$ using $\mathcal{F}$, denoted by $\mathrm{cov}_{\mathcal{F}}(Z)$, is the size of the smallest subfamily $\mathcal{F}'$ of $\mathcal{F}$ such that every vertex of $Z$ is contained in some element of $\mathcal{F}'$.
A family $\mathcal{F}$ of vertex subsets is a \emph{$(b,a)$-container family} of $G$, if $\alpha(F)\leq a$ for every element $F$ in $\mathcal{F}$ and for every vertex subset $Z$ satisfying $\alpha(Z) \leq b$, there exists some element $F$ of $\mathcal{F}$ such that $Z\subseteq F$.

A \emph{Bernoulli random variable} with success probability $p$, is one that takes value $1$ with probability $p$, and $0$ with probability $1-p$. 
We recall a version of the Chernoff bound. A stronger version of this result, along with its full proof, is available in \cite{mitzenmacher2017probability}.
% \todo{should this be in theorem env? {\em I don't think so. I think this is good}}
% \todo[inline]{I think it would help to add the definiton of a Bernoulli random variable}
\begin{proposition}[\cite{hagerup1990guided,mitzenmacher2017probability}]\label{thm:chernoff}
    Let $X_1, X_2, \dots, X_\ell$ be independent Bernoulli random variables, and let $X = \sum_{i \in [\ell]} X_i$. If $R \geq 6\,\mathbb{E}[X]$, then
    $$
    \mathbb{P}[X \geq R] \leq 2^{-R}.
    $$
\end{proposition}

\section{Independence-Containers}\label{sec:containers}
    % A family $\mathcal{F}$ of vertex subsets in a graph $G$ is an \emph{$(b,a)$-container family} of $G$, if $\alpha(F)\leq a$ for every element $F$ in $\mathcal{F}$ and for every vertex subset $H$ satisfying $\alpha(G[H]) \leq b$, there exists some element $F$ of $\mathcal{F}$ such that $H\subseteq F$.
%
For a real $\rho > 1$ and non-negative real $n$ we define $\overline{\log}_\rho(n)$ to be $\log_\rho(n)$ if $n > 0$ and $0$ otherwise.

\begin{lemma}\label{lem:containers1}
There exists an integer $c$ with the following property.
There exists an algorithm that takes as input $(G, \omega, k, b)$ where $\omega$ and $k$ are positive integers, $b \leq \omega$ is a non-negative integer, and $G$ is a $\overline{kK_\omega}$-free graph.
The algorithm runs in time $(n+1)^{  {(2\omega \cdot \overline{\log} (n) + k+ b)^{k+b+1}}} \cdot n^c$ and outputs a
$(b, a)$-container family ${\cal F}$ such that $|{\cal F}| \leq (n+1)^{  {(2\omega \cdot \overline{\log} (n) + k+b)^{k+b+1} }}$ and $a \leq  (2 \omega \cdot \overline{\log}(n) + k+b)^{k+b+1}$.
%of vertex sets in $G$ such that $|{\cal F}| \leq (n+1)^{  {(2\omega \cdot \overline{\log} (n) + k+b)^{k+b+1} }}$, for every set $X \in {\cal F}$ we have $\alpha(X) \leq (2 \omega \cdot \overline{\log}(n) + k+b)^{k+b+1}$, and for every set $S \subseteq V(G)$ such that $\alpha(S) \leq b$ there exists an $X \in {\cal F}$ such that $S \subseteq X$.
\end{lemma} 

\begin{proof}
We begin by describing the algorithm. If $b = 0$ the algorithm outputs ${\cal F} = \{\emptyset\}$. If $b \geq 1$ and $\alpha(G) \leq \omega$ the algorithm outputs ${\cal F} = \{V(G)\}$. Suppose now that $b \geq 1$ and $\alpha(G) > \omega$.
Define $\rho = (1 - \frac{1}{2\omega})^{-1}$, and observe that $1 < \rho \leq 2$. The algorithm considers the set $H$ of all vertices in $G$ of degree at least $n/\rho$. 
We first check by brute force in time $n^{\omega+2}$ whether $H$ contains an independent set $I$ of size $\omega$. 

Suppose first such a set $I$ exists. Observe that $k > 1$, because otherwise $I$ is an induced $\overline{kK_\omega}$ in $G$.
Define $Q = \bigcap_{v \in I}N(v)$, and note that $|Q| \geq n - \sum_{v \in I} |V(G) - N(v)| \geq \frac{n}{2}$ because every vertex in $I$ has degree at least $n(1-\frac{1}{2\omega})$.
We have that $G[Q]$ is $\overline{(k-1)K_\omega}$-free. Indeed, suppose that $G[Q]$ contains a set $Z$ inducing a $\overline{(k-1)K_\omega}$ then $Z \cup I$ is an induced $\overline{kK_\omega}$ in $G$, a contradiction. 
Then, the algorithm calls itself recursively on ($G[Q]$, $\omega$, $k-1$, $b$) and obtains a family ${\cal F}_1$. 
It also calls itself recursively on ($G - Q$, $\omega$, $k$, $b$) and obtains a family ${\cal F}_2$.
The algorithm returns ${\cal F} = {\cal F}_1 \otimes {\cal F}_2$.

If no independent set $I$ of size $\omega$ exists in $H$, then $V(G) - H$ is non-empty since $\alpha(G) > \omega > \alpha(H)$.
The algorithm iterates through every vertex $v \notin H$ and calls itself recursively on ($G[N(v)]$, $\omega$, $k$, $b$) and obtains a family ${\cal F}^v_1$. It also calls itself recursively on  ($G - N[v]$, $\omega$, $k$, $b-1$) and obtains a family ${\cal F}^v_2$.
The algorithm then returns 
$${\cal F} = \{H\} \otimes\left( \bigcup_{v \notin H} \{\{v\}\} \otimes {\cal F}_1^v \otimes {\cal F}_2^v \right).$$
This completes the description of the algorithm. Each recursive call of the algorithm is made on an instance with strictly fewer vertices than $G$. Thus the algorithm always terminates and outputs a non-empty family ${\cal F}$. 
First we show that ${\cal F}$ covers every set with independence number at most $b$.

\begin{claim}\label{clm:family_covers1}
For every set $S \subseteq V(G)$ such that $\alpha(S) \leq b$ there exists an $X \in {\cal F}$ such that $S \subseteq X$.
\end{claim}

%\sta{\label{clm:family_covers1}
%For every set $S \subseteq V(G)$ such that $\alpha(S) \leq b$ there exists an $X \in {\cal F}$ such that $S \subseteq X$.
%}
%

\begin{proof}
We proceed by induction on $|V(G)|$. If $b=0$ then $S=\emptyset$ and the claim holds. 
Similarly, if $\alpha(G) \leq \omega$ then ${\cal F}=\{V(G)\}$ so the claim holds in this case as well.
We now consider the case that $b > 0$ and $\alpha(G) > \omega$.
Suppose first that there exists an independent set $I$ in $H$ of size $\omega$. By the induction hypothesis ${\cal F}_1$ contains a set $X_1$ that contains $S \cap Q$ and similarly ${\cal F}_2$ contains a set $X_2$ that contains $S - Q$. But then $X_1 \cup X_2 \in {\cal F}$ contains $S$. 

Suppose now that no such independent set $I$ exists. 
%If $S \subseteq H$ then every set in ${\cal F}$ contains $H$ and therefore satisfies the conditions of~(\ref{clm:family_covers1}). 
Since every set in ${\cal F}$ contains $H$ we may assume that $S - H$ is non-empty. 
% and therefore satisfies the conditions of~(\ref{clm:family_covers1}). 
Let $v$ be a vertex in $S - H$. 
By the induction hypothesis we have that ${\cal F}_1^v$ contains a set $X_1$ such that $S \cap N(v) \subseteq X_1$. Further, since $S - N[v]$ contains no neighbors of $v$ we have that $\alpha(S - N[v]) \leq \alpha(S) - 1 \leq b - 1$. Thus, by the induction hypothesis 
 ${\cal F}_2^v$ contains a set $X_2$ such that $S - N[v] \subseteq X_2$.
But then $S \subseteq \{v\} \cup X_1 \cup X_2$ and $\{v\} \cup X_1 \cup X_2 \in {\cal F}$, proving the claim.
%~(\ref{clm:family_covers1}). 
\end{proof}

% \medskip
% \todo[inline]{Do we need the paragraph below??}
% {\color{purple}We record two monotonicity facts used in the proofs of Claims~\ref{clm:alphaBound1}
% and~\ref{clm:FBound1}. First, for fixed $\rho > 1$ and fixed $a \geq 0$, the function
% $B_{\rho,a}$ mapping each non-negative integer $n$ to $\binom{\overline{\log}_\rho(n)+a}{a}$, is nondecreasing. Since $\binom{x+a}{a} =
% \prod_{i=1}^{a}\frac{x+i}{i}$ is a finite product of positive strictly increasing
% functions on $[0,\infty)$, it is strictly increasing there, and the claim follows since
% $\overline{\log}_\rho$ is increasing with $\overline{\log}_\rho(n) \geq 0$ for $n \geq 0$. Second, the
% function $B'_{\rho,a}$ mapping each non-negative integer $n$ to $(n+1)^{2B_{\rho,a}(n)-1}$, is nondecreasing. Indeed,
% $B_{\rho,a}(n) \geq \binom{a}{a} = 1$ for $n \geq 0$, so $2B_{\rho,a}(n)-1 \geq 1$,
% and $B_{\rho,a}$ is nondecreasing by the first fact, so for $0 \leq m \leq n$, we have
% $(m+1)^{2B_{\rho,a}(m)-1} \leq (n+1)^{2B_{\rho,a}(m)-1} \leq
% (n+1)^{2B_{\rho,a}(n)-1}$.}
Now we upper bound the independence numbers of the sets in ${\cal F}$.

\begin{claim}\label{clm:alphaBound1}
For every $X \in {\cal F}$ we have that $\alpha(X) \leq 2\omega {\overline{\log}_\rho(n) + k+b \choose k+b} - \omega$.
\end{claim}

\begin{proof}
%\sta{\label{clm:alphaBound1}
%For every $X \in {\cal F}$ we have that $\alpha(X) \leq 2\omega {\overline{\log}_\rho(n) + k+b \choose k+b} - \omega$.
%}
We proceed by induction on $n$. If $b=0$ or $\alpha(G) \leq \omega$ then we have $\alpha(X) \leq \omega \leq 2\omega {\overline{\log}_\rho(n) + k+b \choose k+b} - \omega$. Suppose now that $b \geq 1$ and $\alpha(G) > \omega$. Since $\omega \geq 1$ this implies that $n \geq \alpha(G) \geq 2$. 

If $H$ contains an independent set $I$ of size $\omega$, then by the induction hypothesis we have that for every $X_1 \in {\cal F}_1$ we have $\alpha(X_1) \leq 2\omega {\overline{\log}_\rho(n) + k-1+b \choose k-1+b} - \omega$.
Similarly, by the induction hypothesis (and using the fact that $|Q| \geq \frac{n}{2}$) we have that for every $X_2 \in {\cal F}_2$ we have 
$\alpha(X_2) \leq 2\omega {\overline{\log}_\rho(n)-1 + k+b \choose k+b} - \omega$.
It follows that for every $X \in {\cal F}$  we have 
$$\alpha(X) \leq 
2\omega {\overline{\log}_\rho(n) + k-1+b \choose k-1+b} - \omega +
2\omega {\overline{\log}_\rho(n)-1 + k+b \choose k+b} - \omega
\leq
2\omega {\overline{\log}_\rho(n) + k+b \choose k+b} - \omega.
$$

If $H$ does not contain an independent set $I$ of size $\omega$ then $\alpha(H) \leq \omega - 1$. 
Further, by the induction hypothesis (and using the fact that $|N(v)| < n/\rho$) we have that for every $v \notin H$ and every set $X_1 \in {\cal F}_1^v$ we have 
$\alpha(X_1) \leq 2\omega {\overline{\log}_\rho(n) - 1 + k+b \choose k+b} - \omega$. Additionally, for every $v \notin H$ and every set $X_2 \in {\cal F}_2^v$ we have $\alpha(X_2) \leq 2\omega {\overline{\log}_\rho(n) + k+b-1 \choose k+b-1} - \omega$. Thus, for every $X \in {\cal F}$  we have 
\begin{align*}
\alpha(X) & \leq \omega - 1 + 1 + 2\omega {\overline{\log}_\rho(n) - 1 + k+b \choose k+b} - \omega + 2\omega {\overline{\log}_\rho(n) + k+b-1 \choose k+b-1} - \omega \\
& \leq 2\omega {\overline{\log}_\rho(n) + k+b \choose k+b} - \omega \mbox{.}
\end{align*}
This concludes the proof of the claim.
%proves~(\ref{clm:alphaBound1}). 
\end{proof}

Finally, we upper bound the size of ${\cal F}$.
\begin{claim}\label{clm:FBound1}
$|{\cal F}| \leq (n+1)^{ 2 {\overline{\log}_\rho(n) + k+b \choose k+b} - 1}$.
\end{claim}

\begin{proof}
%\sta{\label{clm:FBound1}
%$|{\cal F}| \leq (n+1)^{ 2 {\overline{\log}_\rho(n) + k+b \choose k+b} - 1}$.}
The proof of this claim closely follows the proof of Claim~\ref{clm:alphaBound1}.
If $b=0$ or $\alpha(G) \leq \omega$ then we have $|{\cal F}| = 1 \leq n+1$. Suppose now that $b \geq 1$ and $\alpha(G) > \omega$. Since $\omega \geq 1$ this implies that $n \geq \alpha(G) \geq 2$. 

If $H$ contains an independent set $I$ of size $\omega$ then, by the induction hypothesis we have that $|{\cal F}_1| \leq (n+1)^{2 {\overline{\log}_\rho(n) + k-1+b \choose k-1+b} - 1}$.
Similarly, by the induction hypothesis (and using the fact that $|Q| \geq \frac{n}{2}$) we have that $|{\cal F}_2| \leq (n+1)^{2 {\overline{\log}_\rho(n)-1 + k+b \choose k+b} - 1}$.
It follows that
$$|{\cal F}| \leq 
(n+1)^{2 {\overline{\log}_\rho(n) + k-1+b \choose k-1+b} - 1} \cdot
(n+1)^{2 {\overline{\log}_\rho(n)-1 + k+b \choose k+b} - 1}
\leq
(n+1)^{ 2 {\overline{\log}_\rho(n) + k+b \choose k+b} - 1 }.
$$

If $H$ does not contain an independent set $I$ of size $\omega$, by the induction hypothesis (and using the fact that $|N(v)| < n/\rho$) we have that for every $v \notin H$ we have 
$|{\cal F}_1^v| \leq (n+1)^{2 {\overline{\log}_\rho(n) - 1 + k+b \choose k+b} - 1}$. Additionally, for every $v \notin H$ we have $|{\cal F}_2^v| \leq (n+1)^{2 {\overline{\log}_\rho(n) + k+b-1 \choose k+b-1} - 1}$. Hence we may conclude that  
\begin{align*}
|{\cal F}| & \leq 
n \cdot
(n+1)^{2 {\overline{\log}_\rho(n) - 1 + k+b \choose k+b} - 1} \cdot
(n+1)^{2 {\overline{\log}_\rho(n) + k+b-1 \choose k+b-1} - 1}\\
& \leq(n+1)^{ 2 {\overline{\log}_\rho(n) + k+b \choose k+b} - 1 } \mbox{.}
\end{align*}
This proves the claim.
%~(\ref{clm:FBound1}). 
\end{proof}

Since $\frac{1}{\log(\rho)} \leq 2\omega$, Claim~\ref{clm:FBound1} implies the claimed size bound on ${\cal F}$, while Claim~\ref{clm:alphaBound1} implies the claimed bound on the independence number of every set $X$ in ${\cal F}$. 
The upper bound of Claim~\ref{clm:FBound1} would apply even if duplicates of the same set in ${\cal F}$ are counted as many times as they are generated by the algorithm. Since the algorithm only spends polynomial time per set in ${\cal F}$ (counting duplicates), the running time bound follows.
\end{proof}

Lemma~\ref{lem:containers1} gives a quasi-polynomial size container family for constant size independent sets. We want to also have quasi-polynomial size container families even for poly-logarithmic size independent sets. The next lemma achieves this.

\begin{lemma}\label{lem:containers2}
There exists an algorithm that takes as input a four-tuple $(G, \omega, k, b)$ where $\omega$ and $k$ are positive integers, $b$ is a non-negative integer, and $G$ is a $\overline{kK_\omega}$-free graph. The algorithm runs in time  $(n+1)^{{(3\omega \cdot \overline{\log} (n) + k)^{k+\omega+1}} \cdot b^{2\omega k}} \cdot n^{O(1)}$
and outputs a $(b, a)$-container family  ${\cal F}$ such that
$|{\cal F}| \leq (n+1)^{{(3\omega \cdot \overline{\log} (n) + k)^{k+\omega+1}} \cdot b^{2\omega k}}$ and
$a \leq (3 \omega \cdot \overline{\log}(n) + k)^{k+\omega+1} \cdot 2^k\cdot b^{\omega k}$
%family ${\cal F}$ of vertex sets in $G$ such that
%\begin{itemize}
%    \item $|{\cal F}| \leq (n+1)^{{(3\omega \cdot \overline{\log} (n) + k)^{k+\omega+1}} \cdot b^{2\omega k}}$,
%    \item for every set $X \in {\cal F}$ we have $\alpha(X) \leq (3 \omega \cdot \overline{\log}(n) + k)^{k+\omega+1} \cdot 2^k\cdot b^{\omega k}$, and
%
%    \item for every set $S \subseteq V(G)$ such that $\alpha(S) \leq b$ there exists an $X \in {\cal F}$ such that $S \subseteq X$.
%\end{itemize}
\end{lemma}

\begin{proof}
%If $a = 0$ the algorithm outputs ${\cal F} = \{\emptyset\}$. Throughout the rest of the algorithm we assume that $a \geq 1$, the algorithm will never make recursive calls with $a = 0$.
%
Let $p$ be the upper bound on the maximum value of $\alpha(X)$ for sets $X$ in the family ${\cal F}$ obtained by the algorithm of Lemma~\ref{lem:containers1} on $(G, \omega, k, \omega)$. In particular $p \leq (3 \omega \cdot \overline{\log}(n) + k)^{k+\omega+1}$.
Similarly let $q$ be the upper bound on the size of the family ${\cal F}$ obtained by the algorithm of Lemma~\ref{lem:containers1} on $(G, \omega, k, \omega)$. In particular $q \leq (n+1)^{  {(3\omega \cdot \overline{\log} (n) + k)^{k+\omega+1} }}$.

We begin by describing the algorithm. If $b \leq \omega$ then the algorithm outputs the family ${\cal F}$ of Lemma~\ref{lem:containers1}. 
The bounds $p$ and $q$ on the independence number and size of the family ${\cal F}$ respectively are below the claimed upper bounds. 
%within the upper bounds claimed in the statement to be proved.\todo{reformulate} 
Throughout the rest of the algorithm we assume that $b > \omega$, as the algorithm never changes $b$ in its recursive calls. Observe that in this case, since $G$ is $\overline{kK_\omega}$-free, we have that $k \geq 2$.
If $\alpha(G) \leq b$, the algorithm outputs ${\cal F} = \{V(G)\}$. Suppose now that $\alpha(G) > b > \omega$. 

The algorithm iterates over every independent set $I$ of $G$ of size at most $b$. For each such independent set $I$, the algorithm proceeds as follows.
First the algorithm iterates over every non-empty subset $Z$ of $I$ of size at most $\omega-1$. 
Let $V_Z^I$ be the set of vertices $v$ in $V(G)$ such that $N(v) \cap I = Z$. The algorithm obtains a family ${\cal F}_Z^I$ by running the algorithm of Lemma~\ref{lem:containers1} on $(G[V_Z^I], \omega, k, \omega)$.
Next the algorithm iterates over every subset $Z$ of $I$ of size exactly $\omega$. Let $V_{\supseteq Z}^I$ be the set of vertices $v$ in $V(G)$ such that $N(v) \cap I \supseteq Z$.
Observe that $Z$ together with a $\overline{(k-1)K_\omega}$ in $G[V_{\supseteq Z}^I]$ would yield a $\overline{kK_\omega}$ in $G$. Hence $G[V_{\supseteq Z}^I]$ is $\overline{(k-1)K_\omega}$-free.
The algorithm obtains a family ${\cal F}_Z^I$ by running itself recursively on $(G[V_{\supseteq Z}^I], \omega, k-1, b)$.
%with same value for $\omega$ and $a$, and with $k=k-1$.
%
Finally the algorithm outputs
$${\cal F} = \bigcup_{I}\left(\{I\} \otimes \bigotimes_{\substack{\emptyset \subset Z \subseteq I ~\text{s.t.} \\ |Z| \leq \omega}} {\cal F}_Z^I \right).$$
Here the union is taken over all non-empty independent sets $I$ of size at most $b$. The algorithm only makes recursive calls on instances with strictly smaller value of $k$, hence it always terminates and outputs a family ${\cal F}$. Next, we show that ${\cal F}$ covers every set of independence number at most $b$.
\begin{claim}\label{clm:family_covers2}
For every set $S \subseteq V(G)$ such that $\alpha(S) \leq b$ there exists an $X \in {\cal F}$ such that $S \subseteq X$.
%\sta{\label{clm:family_covers2}
%For every set $S \subseteq V(G)$ such that $\alpha(S) \leq b$ there exists an $X \in {\cal F}$ such %that $S \subseteq X$.
%}
\end{claim}

\begin{proof}
We proceed by induction on $k$. 
For the base case $k = 1$, $G$ is $\overline{K_\omega}$-free and hence $\alpha(G) < \omega$. If $b \leq \omega$
the statement holds by Lemma~\ref{lem:containers1}. Otherwise $b > \omega > \alpha(G)$, so $\alpha(G) \leq b$
and $\mathcal{F} = \{V(G)\}$, which contains every $S$ with $\alpha(S) \leq b$.

Now assume the statement holds for every $k'<k$ for some arbitrary but fixed positive integer $k >1$.
Similar to the base case, if $b \leq \omega$ or if $\alpha(G) \leq b$ then the statement of the claim holds, therefore assume $\alpha(G) > b > \omega$. 
Let $S$ now be a vertex set such that $\alpha(S) \leq b$. Let $I$ be a maximum size independent set in $S$. 
For each non-empty subset $Z$ of $I$ of size at most $\omega-1$ define $S_Z$ to be the set of vertices $v$ in $S$ such that $N(v) \cap I = Z$. 
By the maximality of $I$, $\alpha(S_Z) \leq |Z| \leq \omega$. Thus, by Lemma~\ref{lem:containers1} there exists a set $X_Z$ in ${\cal F}_Z^I$ such that $S_Z \subseteq X_Z$. 
For each subset $Z$ of $I$ of size exactly $\omega$, let $S_{\supseteq Z}$ be the set of vertices $v$ in $S$ such that $N(v) \cap I \supseteq Z$. Note that $S_{\supseteq Z} \subseteq V_{\supseteq Z}^I$. Thus, by the induction hypothesis there exists a set $X_Z$ in ${\cal F}_Z^I$ such that $S_{\supseteq Z} \subseteq X_Z$. 
By the maximality of $I$ every vertex in $S-I$ has a neighbor in $I$. Thus $I \cup \bigcup_{Z} X_Z$, where the union is taken over all non-empty subsets $Z$
of $I$ of size at most $\omega$, contains $S$ and is an element of ${\cal F}$. This proves the claim. %~(\ref{clm:family_covers2}).
\end{proof}
Next we upper bound the independence number of every set in ${\cal F}$.

\begin{claim}\label{clm:alphaBound2}
For every $X \in {\cal F}$ we have that $\alpha(X) \leq p \cdot 2^k \cdot b^{\omega k}$.
%
%\sta{\label{clm:alphaBound2}
%For every $X \in {\cal F}$ we have that $\alpha(X) \leq p \cdot 2^k \cdot b^{\omega k}$.
%}
\end{claim}

\begin{proof}
We proceed by induction on $k$. 
For the base case $k = 1$, $G$ is $\overline{K_\omega}$-free and hence $\alpha(G) < \omega$. 
Let $X$ be a set in ${\cal F}$.
If $b \leq \omega$, then $\alpha(X) \leq p$ by Lemma \ref{lem:containers1}, and so the claim holds.  
Otherwise $b > \omega > \alpha(G)$, so $\alpha(X) \leq b$ and the claim is true as well.

Now assume the statement holds for every $k'<k$ for some arbitrary but fixed positive integer $k >1$.
Similar to the base case, if $b \leq \omega$ or if $\alpha(G) \leq b$ then the statement of the claim holds, therefore assume $\alpha(G) > b > \omega$.
In this case, $X = I \cup \bigcup_Z X_Z$, where $I$ is an independent set in $G$ of size at most $b$, the union is taken over all non-empty subsets $Z$ of $I$ of size at most $\omega$, and $X_Z \in {\cal F}_Z^I$ for each such $Z$.
For each $Z$ such that $|Z| < \omega$,  Lemma~\ref{lem:containers1} yields that $\alpha(X_Z) \leq p$, which of course is at most $p \cdot 2^{k-1} \cdot b^{\omega (k-1)}$. For each $Z$ such that $|Z| = \omega$,  the induction hypothesis yields that $\alpha(X_Z) \leq p \cdot 2^{k-1} \cdot b^{\omega (k-1)}$.
Since there are at most $b^\omega$ non-empty subsets $Z$ of $I$ of size at most $\omega$, we obtain
$$\alpha(X) \leq b + b^\omega \cdot p \cdot 2^{k-1} \cdot b^{\omega (k-1)} \leq p \cdot 2^k \cdot b^{\omega k}.$$ 
This proves the claim.
\end{proof}

Finally we upper bound the size of ${\cal F}$.

\begin{claim}\label{clm:FBound2}
%\sta{\label{clm:FBound2}
$|{\cal F}| \leq q^{b^{2\omega k}}$.
%}
\end{claim}

\begin{proof}
We proceed by induction on $k$. 
For the base case $k = 1$, $G$ is $\overline{K_\omega}$-free and hence $\alpha(G) < \omega$. 
If $b \leq \omega$ then $|{\cal F}|  \leq q$ by Lemma \ref{lem:containers1}, and so the claim holds.  
Otherwise $b > \omega > \alpha(G)$, so $|{\cal F}| = 1$ and so the claim  is true as well.

Now assume the statement holds for every $k'<k$ for some arbitrary but fixed positive integer $k >1$.
Similar to the base case, if $b \leq \omega$ or if $\alpha(G) \leq b$ then the statement of the claim holds, therefore assume $\alpha(G) > b > \omega$. 
Otherwise ${\cal F}$ is the union over at most $n^b$ products (one for each choice of $I$), where each family ${\cal F}_Z^I$ in the product has size at most $q$ (by Lemma~\ref{lem:containers1}, if $|Z| < \omega$) or at most $q^{b^{2\omega (k-1)}}$ (by the induction hypothesis, if $|Z| = \omega$). Since $k,b,\omega$ are positive integers we have that $q \leq q^{b^{2\omega (k-1)}}$. Further, as there are at most $b^\omega$ choices for $Z$ it follows that 
$$|{\cal F}| \leq n^b \cdot (q^{b^{2\omega (k-1)}})^{b^\omega} \leq ((q^{b^{2\omega (k-1)}})^{b^\omega})^2 \leq q^{b^{2\omega k}}.$$
This proves the claim.
\end{proof}

Putting everything together, Claim~\ref{clm:FBound2} implies the claimed size bound on ${\cal F}$, while Claim~\ref{clm:alphaBound2} implies the bound on the independence number of every set $X$ in ${\cal F}$. 
The upper bound of Claim~\ref{clm:FBound2} would apply even if duplicates of the same set in ${\cal F}$ are counted as many times as they are generated by the algorithm. Since the algorithm only spends polynomial time per set in ${\cal F}$ (counting duplicates), the running time bound follows.    
\end{proof}

Lemma~\ref{lem:containers2} immediately implies that for every hereditary class ${\cal C}$, if there exists a positive integer $k$ such that $\overline{kK_k} \notin {\cal C}$, then for every graph $G \in {\cal C}$ and every integer $b$ which is polynomial in $\log n$, $G$ contains a $(b, a)$-container family ${\cal F}$ of quasi-polynomial size, with $a$ polynomial in $\log n$. 
We complete the section by showing that the restriction that  ${\cal C}$ is $\overline{kK_k}$-free is necessary. In particular, in every hereditary class ${\cal C}$ of graphs that contains $\overline{kK_k}$ for every integer $k \geq 0$, even $(1, a)$-container families require size at least $2^{n/2a}$, as can be observed by applying Lemma~\ref{lem:containerLowerBoundGraph} with $k = \frac{n}{2a}$ and $\omega = 2a$.

\begin{lemma}\label{lem:containerLowerBoundGraph}
Let $a, b, k, \omega$ be positive integers such that $b \leq a < \omega$ and ${\cal F}$ be a $(b, a)$-container family for $\overline{kK_\omega}$. Then 
$|{\cal F}| \geq {\omega \choose b}^k / {a \choose b}^k$.
\end{lemma}

%\todo{abuse notation to say $\alpha(F)$ for $\alpha(G[F])$. Add to preliminaries, or is this a problem?\\\vspace*{10pt} {\em I vote for ``add to the preliminaries"\\}\vspace*{10pt} \emph{added}}
\begin{proof}
Let $V(\overline{kK_\omega}) = V_1 \cup \ldots \cup V_k$ where $V_i$ is the vertex set of the $i$'th copy of 
$\overline{K_\omega}$.
Let ${\cal Q}$ be the family of all subsets $Q$ of $V(\overline{kK_\omega})$ such that $|Q \cap V_i| = b$ for every $i\in[k]$.
It follows that $|{\cal Q}| = {\omega \choose b}^k$ and that $\alpha(Q) = b$ for every $Q \in {\cal Q}$.
Let now ${\cal F}$ be a $(b, a)$-container family for $\overline{kK_\omega}$.
For every $F \in {\cal F}$ and every $i\in[k]$ it holds that $|F \cap V_i| \leq \alpha(F) \leq a$. Hence $F$ covers at most ${a \choose b}^k$ sets in ${\cal Q}$. 
Thus $|{\cal F}| \geq {\omega \choose b}^k / {a \choose b}^k$, as claimed. 
\end{proof}

We are now in position to prove Theorem~\ref{thm:containerIff}.

\MainThmContainer*
\begin{proof}
    Let ${\cal C}$ be a hereditary class of graphs, and $b : \mathbb{N} \rightarrow \mathbb{N}$ be a function such that $b(n) = O((\log n)^{O(1)})$.
    
    Suppose first that there exists a positive integer $k$ such that  $\overline{kK_k} \notin {\cal C}$, and let
    $a(n) = (6k\cdot \overline{\log}(n))^{2k+1}\cdot 2^k \cdot (b(n))^{k^2}$. Note that $a(n)$ is a poly-logarithmic function.
    Let $G$ be an arbitrary graph in ${\cal C}$, $n = |V(G)|$, and ${\cal F}$ be the $(b(n), a(n))$-container family for $G$ guaranteed by Lemma~\ref{lem:containers2}. Then ${\cal F}$ has quasi-polynomial size. This concludes the proof of the implication {\em (ii) $\rightarrow$ (i)}.
    
    For the implication {\em (i) $\rightarrow$ (ii)}, suppose for contradiction that there exists a poly-logarithmic function $a : \mathbb{N} \rightarrow \mathbb{N}$ and a quasi-polynomial function $q : \mathbb{N} \rightarrow \mathbb{N}$ such that for every $n$-vertex graph $G \in {\cal C}$, $G$ has a $(b(n), a(n))$-container family ${\cal F}$ of size at most $q(n)$, and that $\overline{kK_k} \in {\cal C}$ for every positive integer $k$.
    Let $n$ be a sufficiently large perfect square such that $2a(n) \leq \sqrt{n}$ and $q(n) < 2^{\sqrt{n}}$, and set $k = \sqrt{n}$.
    Let ${\cal F}$ be the $(b(n), a(n))$-container family of $\overline{kK_k}$ of size at most $q(n)$.
    By Lemma~\ref{lem:containerLowerBoundGraph} we have that 
    $$|{\cal F}| \geq \left(\frac{{k \choose b(n)}}{{a(n) \choose b(n)}}\right)^k \geq 2^k = 2^{\sqrt{n}} > q(n).$$
    This contradicts that $|{\cal F}| \leq q(n)$.
\end{proof}

%Lemma~\ref{lem:containerLowerBoundGraph} yields a lower bound on the size of containers in a particular graph. We can leverage Lemma~\ref{lem:containerLowerBoundGraph} to show that every class that contains arbitrarily large $\overline{kK_k}$ contain graphs where every container family has exponential size. 

%In particular, if ${\cal C}$ be a hereditary class of graphs such that $\overline{kK_k} \in {\cal C}$ for every $k \geq 1$ then we also have that
%
%$\overline{kK_\omega} \in {\cal C}$ for every $k \geq 1$ and $\omega \geq 1$, because $\overline{kK_\omega}$ is an induced subgraph of $\overline{k'K_{k'}}$ where $k' = \max(k, \omega)$.
%
%Let $a \geq 1$ and $k$ be positive integers and $G = \overline{kK_{2a}}$. We have that $n = 2ak$, and Lemma~\ref{lem:containerLowerBoundGraph} yields that every $(1,a)$-container family in $G$ has size at least ${2a \choose 1}^k / {a \choose 1}^k = 2^k = 2^{n/2a}$.

%\begin{lemma}\label{lem:containerLowerBoundClass}
%Let ${\cal C}$ be a hereditary class of graphs such that $\overline{kK_k} \in {\cal C}$ for every {\cal C} and $a$ be an integer. 
%\end{lemma}

 \section{$A$--$B$ separators with Small Independence Number}\label{sec:ab_seps}
    In this section we prove the following theorem:

\begin{restatable}{theorem}{ABSeparatorRounding}
\label{thm:A-B_separator}
Let $G$ be a graph, $a$ be a positive integer, $\mathcal{F}$ be a family of vertex subsets such that $\alpha(F) \leq a$ for every $F\in\mathcal{F}$, $A, B \subseteq V(G)$ be  vertex subsets and $\{x_F\}_{F \in \mathcal{F}}$ be a fractional $(A,B)$-separator in $G$. Then, there exists an $A$--$B$ separator $S$ in $G$ such that $\mathrm{fcov}_{\mathcal{F}}(S)$ is at most $12 \cdot \log 2n \cdot a \cdot \sum_{F \in \mathcal{F}} x_F$.
\end{restatable}

\begin{proof}
    We define the set of \emph{heavy} vertices to be $ \mathbf{V} := \bigl\{ v \in V(G) \,\mid\, \sum_{\substack{F \in \mathcal{F} \\ F \ni v}} x_F \ge \frac{1}{2n} \bigr\}$. For any walk $Q$ in $G$, we use $\mathbf{Q}$ to denote $V(Q) \cap \mathbf{V}$.
    For every vertex $v \in V(G)$, define
    $$
    x_v :=
    \begin{cases}
    \min\!\left\{ 1,\; \sum_{\substack{F \in \mathcal{F} \\ F \ni v}} 2x_F \right\} & \text{if } v \in \mathbf{V},\\[4pt]
    0 & \text{otherwise}.
    \end{cases}
    $$
    By definition, $x_v \ge \frac{1}{n}$ whenever $x_v > 0$.
    For each $v \in V(G)$, let
    $$
    d_v := \min \left\{ 1, \min\left\{ \sum_{\substack{F \in \mathcal{F} \\ F \cap \mathbf{P} \neq \emptyset}} 2x_F \;\middle|\; P \text{ is an } A\text{--}v \text{ path} \right\} \right\},
    \qquad
    I_v := \left[ d_v - x_v,\; d_v \right].
    $$
    Observe that for every $A$--$v$ walk $Q$, $d_v \le \sum_{\substack{F \in \mathcal{F} \\ F \cap \mathbf{Q} \neq \emptyset}} 2x_F$, since $G[Q]$ contains an $A$--$v$ path. We begin with the following claim.
    \begin{claim}\label{claim:intervals_overlap}
        Let $u,v \in V(G)$ be two vertices such that there exists a $u$--$v$ path $P$ in $G$ that contains no heavy vertices except possibly $u$ or $v$. Then $d_v \le d_u + x_v$.
    \end{claim}
    \begin{proof}
        Let $Q$ be the $A$--$v$ walk obtained by appending $P$ to an $A$--$u$ path realizing $d_u$. Then,
        \begin{align*}
        d_v
        &\ =\ \min \left\{ 1,\, \min\left\{
            \sum_{\substack{F \in \mathcal{F} \\ F \cap \mathbf{P'} \neq \emptyset}} 2x_F
            \;\middle|\;
            P' \text{ is an } A\text{--}v \text{ path}
           \right\} \right\} \\
        &\ \leq\ \min \left\{1,\,
            \sum_{\substack{F \in \mathcal{F} \\ F \cap \mathbf{Q} \neq \emptyset}} 2x_F
            \right\}\\
        &\ \leq\ \min \left\{1,\, d_u +
            \begin{cases}
                \displaystyle
                \sum_{\substack{F \in \mathcal{F} \\ F \ni v}} 2x_F & \text{if } v \in \mathbf{V}, \\[6pt]
                0 & \text{otherwise}
            \end{cases}
            \right\}
         \ =\ \min\{1,\, d_u + x_v\}
         \ \leq\ d_u + x_v.
    \end{align*}

    Here the first inequality follows from the fact that $Q$ is an $A$--$v$ walk and the equality in the last line follows from the fact that $\min\{1,\min\{1,a\}+b\} = \min\{1, a+b\}$ for all non-negative real numbers $a,b$.
    \end{proof}
    
    Let $uv\in E(G)$, applying Claim~\ref{claim:intervals_overlap} to the pair $(u,v)$ and $(v,u)$ implies that $I_u \cap I_v \neq \emptyset$.
    For every $r \in [0,1]$, we define $S_r = \{ v \in \mathbf{V} \mid r \in I_v \}$ and make the following claims. 
    \begin{claim}\label{claim:S_r_is_A-B_sep}
        $S_r$ forms an $A$--$B$ separator for every $r\in [0,1]$
    \end{claim}
    
    \begin{proof}
        Note that if $v\in A$, then by definition $d_v = x_v$. We also show that if $v\in B$, then $d_v \geq 1$. To see this, let $P$ be an $A$--$v$ path in $G$ and observe that,
        $$
        \sum_{\substack{F \in \mathcal{F}\\ F \cap \mathbf{P} \neq \emptyset}} 2x_F \quad
        \geq \quad 2\bigl( 1 - \sum_{\substack{ F \in \mathcal{F} \\ F \cap P \neq \emptyset \\ F \cap \mathbf{P} = \emptyset}} x_F \bigr)\quad
        \geq \quad 2\bigl( 1 - \sum_{v \in P\setminus \mathbf{P}}\sum_{F \ni v} x_F \bigr) \quad
        \geq \quad  1.
        $$
        Here, the first transition uses that $\{x_F\}_{F\in\mathcal{F}}$ is a fractional $(A,B)$-separator in $G$, and therefore $\sum_{\substack{ F \in \mathcal{F} \\ F \cap P \neq \emptyset}}x_F \geq 1$. The last transition uses that $V(P)\setminus \mathbf{P}$ has at most $n$ vertices and that $\sum_{F \ni v} x_F$ is at most $\frac{1}{2n}$ for every $v \in V(P)\setminus \mathbf{P}$. 
        
        Let $r \in [0,1]$ and let $P = (v_1, v_2, \dots, v_p)$ be an $A$--$B$ path in $G$. For vertices $v_i,v_j \in V(P)$, we use $P_{v_i,v_j}$ to denote the subpath of $P$ from $v_i$ to $v_j$. Also, we use $\{v'_1, \dots, v'_\ell\}$ to denote the vertices of $\mathbf{P}$, listed in order of their appearance on $P$.
        Let $v'_i$ be the first vertex in $\mathbf{P}$ such that $d_{v'_i} \ge r$. We now argue that such a vertex exists. As $v_p \in B$ implies $d_{v_p} \ge 1$, if $v_p\in \mathbf{P}$ then we are done. Otherwise, as $P_{v'_\ell, v_p}$ is a $v'_\ell$--$v_p$ path containing no heavy vertices other than $v'_\ell$, Claim~\ref{claim:intervals_overlap} applied to $(v'_\ell, v_p)$ gives $d_{v'_\ell} \ge d_{v_p} \ge 1 \ge r$. Therefore $v'_i$ is well defined.
        Let us consider the case when $i=1$. Since $v_1 \in A$ and $P_{v_1, v'_1}$ is a $v_1$--$v'_1$ path containing no heavy vertices other than $v'_1$, Claim~\ref{claim:intervals_overlap} applied to $(v_1, v'_1)$ gives $[0,r] \subseteq I_{v'_1}$, and therefore $v'_i = v'_1 \in S_r$. 
        Otherwise, consider $v'_{i-1}$. Since $d_{v'_{i-1}} < r$ and $P_{v'_{i-1}, v'_i}$ contains no heavy vertices other than $v'_{i-1}$ and $v'_i$, Claim~\ref{claim:intervals_overlap} applied to $(v'_{i-1}, v'_i)$ gives $d_{v'_i} - x_{v'_i} \le d_{v'_{i-1}} < r$, which implies $r \in I_{v'_i}$, and hence $v'_i \in S_r$.
        Therefore $S_r$ is an $A$--$B$ separator.
    \end{proof}

    Consider the random process where an $r\in [0,1]$ is sampled uniformly at random.
    
    \begin{claim}\label{claim:S_r_has_small_cover}
        $\mathbb{E}_r(\mathrm{fcov}_{\mathcal{F}}(S_r))\ \leq\ 12 \cdot \log 2n \cdot a \cdot \sum_{F \in \mathcal{F}} x_F$
    \end{claim}
    
    \begin{proof}
        For each $F\in \mathcal{F}$ and $r\in [0,1]$ we define 
        $$
        m^r_F \ := \
        \begin{cases}
            \quad 0 & \text{if } F \cap S_r = \emptyset,\\
            \quad \max\Bigl\{ \frac{1}{x_v} \, |\, v \in F \cap S_r \Bigr\} & \text{otherwise}.
        \end{cases}
        $$    
        Since $S_r \subseteq \mathbf{V}$, we have $0 < \frac{1}{n} \leq x_v \leq 1$ for every $v \in F \cap S_r$, and thus $m^r_F$ is well defined and satisfies $1 \leq m^r_F \leq n$ whenever it is non-zero.
        Let $\hat{y}_F^r = 2x_F \cdot m^r_F$. Note that if $u \in S_r$, then
        $$
        \sum_{\substack{F \in \mathcal{F}\\ F \ni u}} \hat{y}_F^r \quad  
        \geq \quad \sum_{\substack{F \in \mathcal{F}\\ F \ni u}} 2x_F \cdot \max\left\{ \frac{1}{x_v} \,\middle|\, v \in F \cap S_r \right\} \quad 
        \geq \quad \frac{1}{x_u} \cdot \sum_{\substack{F \in \mathcal{F}\\ F \ni u}} 2x_F \quad 
        \geq \quad 1.
        $$
        Hence, for every $r \in [0, 1]$ we have that $\{\hat{y}_F^r\}_{F \in \mathcal{F}}$ forms a fractional cover of $S_r$. Therefore, 
        \begin{align}\label{eq:expectaion_A-B_separator}
            \nonumber \mathbb{E}_r[\mathrm{fcov}_{\mathcal{F}}(S_r)] \quad
            &\leq \quad \mathbb{E}_r\left[\sum_{F \in \mathcal{F}} \hat{y}_F^r\right]\quad = \quad \sum_{F \in \mathcal{F}} 2x_F \cdot \mathbb{E}_r[m_F^r]\\
            \nonumber &\leq \quad \sum_{F \in \mathcal{F}} 2x_F \left[ \sum_{i = 0}^{\infty} 2^{i+1} \cdot \mathbb{P}_r[2^i \leq m_F^r < 2^{i+1}] \right]\\
            &\leq \quad \sum_{F \in \mathcal{F}} 2x_F \left[ \sum_{i = 0}^{\lfloor \log n \rfloor} 2^{i+1} \cdot \mathbb{P}_r[2^i \leq m_F^r] \right].
        \end{align} 
        Note that the last and penultimate inequalities hold because whenever $m_F^r$ is non-zero, its value is at least $1$ and at most $n$ for every $F \in \mathcal{F}$ and $r \in [0,1]$, as observed earlier.
        
        Consider $F \in \mathcal{F}$. For each $0 \leq i \leq \lfloor \log n \rfloor$, let $F_i = \{v\in F\ |\ x_v\leq 2^{-i}\}$ and define $\Hat{I}_v^i = [d_v-x_v-2^{-i},\ d_v+2^{-i}]\cap [0,1]$ for every $v \in F_i$. Let $J$ be a maximal independent set of $F_i$. Since $J$ is also an independent set in $F$, we have $|J| \leq a$. Moreover, because $J$ is maximal, for every $v \in F_i\setminus J$ there exists some $u \in J$ such that $uv\in E(G)$ and hence $I_v \cap I_u \neq \emptyset$. Since $x_v \leq 2^{-i}$, it follows that $I_v \subseteq \Hat{I}_u^i$. Now, we have:
        $$
        \mathbb{P}_r[2^i \leq m_F^r] \quad
        = \quad \mathbb{P}_r\left[\exists v \in F_i\ \middle|\ r \in I_v \right] \quad
        \leq \quad \mathbb{P}_r\left[\exists u \in J\ \middle|\ r \in \Hat{I}_u^i \right] \quad
        \leq \quad \sum_{u \in J} \lVert \Hat{I}_u^i \rVert \quad
        \leq \quad  3 \cdot 2^{-i} \cdot a,
        $$
        where for an interval $I = [z_1, z_2] \subseteq [0,1]$, $\lVert I \rVert$ denotes $z_2 - z_1$. 
        Combining this with Equation~\ref{eq:expectaion_A-B_separator}, we conclude that 
        $$
        \mathbb{E}_r[\mathrm{fcov}_{\mathcal{F}}(S_r)] \quad
        \leq \quad \sum_{F \in \mathcal{F}} 2x_F \left[ \sum_{i = 0}^{\lfloor \log n \rfloor} 2^{i+1} \cdot 3 \cdot 2^{-i} \cdot a \right] \quad
        \leq \quad 12 \cdot \log 2n  \cdot a \cdot \sum_{F \in \mathcal{F}} x_F.
        $$
    \end{proof} 
    
    Now, as $\mathbb{E}_r[\mathrm{fcov}_{\mathcal{F}}(S_r)] \leq 12 \cdot \log 2n \cdot a \cdot \sum_{F \in \mathcal{F}} x_F$ by Claim~\ref{claim:S_r_has_small_cover}, there must exist some $r_o\in[0,1]$ such that the fractional cover of $S_{r_o}$ using sets from $\mathcal{F}$ must have value at most $12 \cdot \log 2n \cdot a \cdot \sum_{F \in \mathcal{F}} x_F$. Furthermore, $S_{r_o}$ is an $A$--$B$ separator by Claim~\ref{claim:S_r_is_A-B_sep}, which concludes our proof.
\end{proof}

\section{Path Packing via the Dual $A$--$B$ Separator LP}\label{sec:ab_sampling}
% \section{Obstructions to $A$--$B$ Separators with Small Independence \\Number}\label{sec:ab_sampling}
    The main theorem that we prove in this section is the following:

\begin{theorem}
\label{thm:A-B_path_packing}
Let $G$ be a graph, $A, B \subseteq V(G)$ be vertex subsets, $\mathcal{F}$ be a family of vertex subsets, and $f$ be the minimum of $\sum_{F \in {\cal F}} x_F$ over all fractional $(A, B)$-separators $\{x_F\}_{F\in{\cal F}}$. Then, for every $\ell \geq \log 2|{\cal F}|$ there exists a multiset $\mathcal{Q}$ of induced $A$--$B$ paths in $G$ of cardinality at least $f \cdot \ell$, such that for every $F \in \mathcal{F}$, the number of paths in $\mathcal{Q}$ that have a non-empty intersection with $F$ is at most $6(\ell+1)$.
\end{theorem}

\begin{proof}
    Assume that $f$ is non-zero, since otherwise the theorem holds trivially.
    Let $\mathcal{P}$ denote the set of all induced $A$--$B$ paths in $G$. We describe the \emph{$A$--$B$ separator linear program}, using non-negative real variables $\{x_F\}_{F\in\mathcal{F}}$.
    \begin{align}\label{lp:A-B_separator}
        \text{Minimize :}&\quad \sum\limits_{F\in\mathcal{F}} x_{F}\\
        \nonumber \text{subject to :}&\quad \sum\limits_{\substack{F\in \mathcal{F}\\ F\cap P\neq \emptyset}} x_{F} \geq 1 && \forall P\in \mathcal{P}
    \end{align}
    
    Observe that $f$ equals the optimal value of the above linear program and is at least 1. This is because, if the constraint $\sum\limits_{\substack{F\in \mathcal{F}\\ F\cap P\neq \emptyset}} x_{F} \geq 1$ holds for every induced $A$--$B$ path $P$, then it also holds for every $A$--$B$ path, since for any such path $P$ in $G$, there exists an induced $A$--$B$ path contained within $G[P]$. By strong duality~\cite{doi:10.1137/1025101}, the dual of LP~(\ref{lp:A-B_separator}) also has the same optimum value. We describe the dual using non-negative real variables $\{y_P\}_{P\in\mathcal{P}}$. 
    \begin{align}\label{lp:path_packing}
        \text{Maximize :}&\quad \sum\limits_{P\in\mathcal{P}} y_{P}\\
        \nonumber \text{subject to :}&\quad \sum\limits_{\substack{P\in \mathcal{P}\\ P\cap F\neq \emptyset}} y_{P} \leq 1 && \forall F\in \mathcal{F}
    \end{align}
    
    Let $\{y_P\}_{P\in\mathcal{P}}$ be the values assigned to the corresponding variables in an arbitrary but fixed optimum solution of the dual. Observe that $\mathcal{D} := \{\,y_P/f\,\}_{P\in\mathcal{P}}$ is a probability distribution over $\mathcal{P}$: each $y_P$ is non-negative and $\sum_{P\in\mathcal{P}}y_P = f$. Let $\mathcal{Q} := \{Q_i\}_{i\in[\,f \cdot \ell\,]}$ be $\lceil f \cdot \ell \rceil$ independent samples drawn from the distribution $\mathcal{D}$ over the set $\mathcal{P}$. Let $F \in \mathcal{F}$, and let $\chi_F$ denote the number of paths in $\mathcal{Q}$ that intersect $F$. For every $i\in[\,f \cdot \ell\,]$, we have that,
    $$
    \displaystyle \mathop{\mathbb{P}}\limits_{\substack{Q_i\sim\mathcal{D}}}\left[\,Q_i \cap F \neq \emptyset\,\right] \quad
    = \quad \sum_{\substack{P\in\mathcal{P}\\ P\cap F\neq \emptyset}}
    \displaystyle \mathop{\mathbb{P}}\limits_{\substack{Q_i\sim\mathcal{D}}}\left[\, Q_i = P\, \right] \quad
    = \quad \sum_{\substack{P\in\mathcal{P}\\ P\cap F\neq \emptyset}} \frac{y_{P}}{f} \quad
    \leq \quad \frac{1}{f}.
    $$
    The bound $f \geq 1$, together with linearity of expectation, implies that the expected value of $\chi_F$ is at most $\ell + 1$. Let $\chi_{\mathcal{F}} := \max\{\chi_F\, |\, F\in\mathcal{F}\}$. Applying union bound over all $F\in\mathcal{F}$ and the Chernoff bound from Proposition \ref{thm:chernoff} to $\chi_F$ for every $F\in\mathcal{F}$, we get :
    $$
    \mathbb{P}\left[\,\chi_{\mathcal{F}}\; <\; 6(\ell+1)\,\right] \quad
    \geq \quad 1\ -\ \sum_{F\in\mathcal{F}}\mathbb{P}\left[\,\chi_F\; \geq\; 6(\ell+1)\,\right] \quad
    \geq \quad 1\ -\ |\mathcal{F}| \cdot 2^{-6(\ell+1)} \quad
    \geq \quad \frac{3}{4}.
    $$
    Hence the sampled multiset has cardinality at least $f \cdot \ell$, consists of induced $A$--$B$ paths in $G$, and with probability at least $\frac{3}{4}$, satisfies that for every $F \in \mathcal{F}$, at most $6(\ell+1)$ of these paths intersect $F$. Therefore, it follows that there  exists some $\mathcal{Q}$ with these properties, which concludes the proof of the theorem. 
\end{proof}

\section{Balanced Separators with Small Independence Number}\label{sec:balanced_separator}

We begin this section by defining the \emph{balanced separator linear program}. Let $G$ be a graph, $I \subseteq V(G)$ an independent set, and $\mathcal{F}$ a family of vertex subsets of $G$. For each pair $u, v \in I$, let $\mathcal{P}_{u,v}$ denote the set of all induced paths from $u$ to $v$ in $G$. We describe the \emph{balanced separator linear program} corresponding to the instance $(G, I, \mathcal{F})$, using non-negative real variables $x_F$ and $d_{u,v}$, defined for every $F \in \mathcal{F}$ and $u, v \in I$.

\begin{align}\label{lp:container_cover}
    \text{Minimize:} \quad & \sum_{F \in \mathcal{F}} x_F \\
    \nonumber
    \text{subject to:} \quad 
    & \sum_{v \in I} d_{u,v} \geq \frac{|I|}{10} && \forall u \in I \\
    \nonumber
    & d_{u,v} \leq \sum_{\substack{F \in \mathcal{F} \\ F \cap P \neq \emptyset}} x_F && \forall u,v \in I,\ P \in \mathcal{P}_{u,v} \\
    \nonumber
    & d_{u,v} \leq 1 && \forall u,v \in I
\end{align}

\begin{observation}\label{obs:lp_is_feasible}
    For every graph $G$, independent set $I \subseteq V(G)$, positive integer $a\geq 2$ and family of vertex subsets $\mathcal{F}$ satisfying (1) $\alpha(F) \leq a$ for every $F\in\mathcal{F}$, (2) $V(G)\subseteq \bigcup_{F\in\mathcal{F}}F$, the balanced separator linear program described in LP~(\ref{lp:container_cover}) is feasible and has optimum objective function value at most $|\mathcal{F}|$.
\end{observation}

The above observation follows from the fact that the assignment $x_F = 1$ and $d_{u,v} = 1$ for every $F\in \mathcal{F}$ and $u,v\in I$ is valid and bounded. Indeed, the second constraint is satisfied since every induced path from $u$ to $v$ in $G$, for each $u,v\in I$, intersects at least one $F\in \mathcal{F}$ by property~(2). Furthermore, the first and third constraints are satisfied by definition, since $d_{u,v} = 1$ and $\sum_{v \in I} d_{u,v} = |I|$ for every $u\in I$.
We are now ready to state the main theorem proved in this section:

\begin{restatable}{theorem}{LPoptVSintegral}
\label{thm:balanced_separator_rounding}
For every graph $G$, independent set $I \subseteq V(G)$, positive integer $a\geq 2$ and family of vertex subsets $\mathcal{F}$ satisfying (1) $\alpha(F) \leq a$ for every $F\in\mathcal{F}$, (2) $V(G)\subseteq \bigcup_{F\in\mathcal{F}}F$, there exists an $(I, \frac{95}{100})$-balanced separator $S$ in $G$ such that
$$
\mathrm{fcov}_{\mathcal{F}}(S) \leq 17000 \cdot \log 2n \cdot a^2 \cdot \log(a \cdot \lpopt + 4) \cdot \lpopt,
$$
% where $\lpopt$ is the optimum value of the balanced separator linear program corresponding to $(G, I, \mathcal{F})$.
whenever $\lpopt$, which denotes the optimal value of the balanced separator linear program corresponding to $(G, I, \mathcal{F})$, is positive.
\end{restatable}

For the remainder of this section, we fix a graph $G$, an independent set $I \subseteq V(G)$, a positive integer $a\geq 2$, and a family of vertex subsets $\mathcal{F}$ satisfying properties (1) and (2) described in the theorem. We also fix an optimal solution to the balanced separator linear program corresponding to $(G, I, \mathcal{F})$, which exists and is bounded via Observation~\ref{obs:lp_is_feasible}. Let $x_F$ and $d_{u,v}$, for every $F \in \mathcal{F}$ and $u,v\in I$, denote the values assigned to the corresponding variables in this solution. Let $\lpopt$ denote the objective function value of this solution, and assume $\lpopt > 0$.

For every $v \in V(G)$, define $x_v := \sum_{\substack{F \in \mathcal{F} \\ F \ni v}} x_F$. Also, for every $u, v \in V(G)$, we define $d(u,v) := \min\{ \sum_{\substack{F \in \mathcal{F} \\ F \cap P \neq \emptyset}} x_F \,|\, P \in \mathcal{P}_{u,v} \}$. Now, we make some observations about the function $d$.

\begin{observation}\label{obs:distance_defintions_match}
    If $u, v \in I$, then $\min\{d(u,v), 1\} \geq d_{u,v}$.
\end{observation}

\begin{observation}\label{obs:triangle_inequality}
    If $u,v,w \in V(G)$, then $d(u,w) \leq d(u,v) + d(v,w)$. Furthermore, if $vw \in E(G)$ and $u \in V(G)$, then $d(u,w) \leq d(u,v) + x_w$. 
\end{observation}

Observation~\ref{obs:distance_defintions_match} follows immediately from the linear program constraints: \emph{$d_{u,v} \leq \sum_{\substack{F\in \mathcal{F}\\ F\cap P \neq\emptyset}} x_F$ and $d_{u,v} \leq 1$ for every $u,v\in I$ and $P\in\mathcal{P}_{u,v}$}. Furthermore, note that $d(u,v) \leq \sum_{\substack{F \in \mathcal{F} \\ F \cap Q \neq \emptyset}} x_F$ for any $u$--$v$ walk $Q$ in $G$, as we can always find an induced $u$--$v$ path in $G[Q]$. Hence, Observation~\ref{obs:triangle_inequality} follows from the fact that if $P_1$ and $P_2$ are induced $u$--$v$ and $v$--$w$ paths realizing $d(u,v)$ and $d(v,w)$ respectively, then appending $P_2$ to $P_1$ yields a $u$--$w$ walk $Q$ satisfying $\sum_{\substack{F \in \mathcal{F} \\ F \cap Q \neq \emptyset}} x_F \leq d(u,v) + d(v,w)$. Similarly, if $vw \in E(G)$, then appending $w$ to the $u$--$v$ path realizing $d(u,v)$ produces a $u$--$w$ walk satisfying $\sum_{\substack{F \in \mathcal{F} \\ F \cap Q \neq \emptyset}} x_F \leq d(u,v) + x_w$.\\

Let $\epsilon := \frac{1}{1300 \cdot \log(a \cdot \lpopt + 4) }$, 
% Let $\epsilon := \frac{1}{{\color{blue} 800 \cdot  \log(a \cdot \lpopt + 4) }}$, 
$\ell_{\max} := 2 \cdot  \lceil \log(a \cdot \lpopt + 4) \rceil + 12$ and define $r_{i} := (4i - 2)\epsilon$, where $i$ is a positive integer. Let $Z_0 := \{v \in V(G) \mid x_v \geq \epsilon\}$. Furthermore, for every vertex $u \in V(G)$, subgraph $C\subseteq G$ and positive real number $r$, define 
$$
B_C(u, r) := \{v \in V(C) \mid d(u, v) \leq r\}
\quad \text{and} \quad
\delta_C(u, r) := B_C(u, r + 4\epsilon) \setminus B_C(u, r + \epsilon).
$$

\begin{lemma}\label{lem:good_layer_many_outside}
    Let $Z \subseteq V(G)$ be a superset of $Z_0$ and let $\ell \in [\,\ell_{\max}\,]$. Let $C$ be a connected component of $G - Z$ with $\Bar{u} \in I \cap V(C)$. Then,
    $$
    \left| I \setminus B_C(\Bar{u}, r_{\ell+1}) \right| \geq \frac{5|I|}{100}.
    $$
\end{lemma}
\begin{proof}
    Since $\ell \in [\,\ell_{\max}\,]$, we have 
    $$
    r_{\ell+1} \
    \leq \ r_{\ell_{\max} + 1} \
    < \ (8\cdot \log(a \cdot \lpopt + 4) + 58)\epsilon \
    \leq \ 37\cdot \log(a \cdot \lpopt + 4)\cdot \epsilon \
    < \ \frac{1}{20}.
    $$
    Here, the first inequality substitutes the definitions and uses the fact that $\lceil x\rceil < x+1$ for every real number $x$, while the second inequality uses the fact that $\log(x + 4) \geq 2$ for every non-negative real number $x$.
    Now assume, for the sake of contradiction, that $\left| I \setminus B_C(\Bar{u}, r_{\ell+1}) \right| < \frac{5|I|}{100}$.
    Then we obtain:
    $$
    \begin{aligned}
    \sum_{v \in I} \min\{d(\Bar{u}, v),1\} \quad
    &= \quad \sum_{v \in I \cap B_C(\Bar{u}, r_{\ell+1})} \min\{d(\Bar{u}, v),1\} 
       \quad + \sum_{v \in I \setminus B_C(\Bar{u}, r_{\ell+1})} \min\{d(\Bar{u}, v),1\} \quad \\
    &\leq \quad\left(|I| - |I\setminus B_C(\Bar{u}, r_{\ell+1})|\right)\cdot r_{\ell+1} + |I\setminus B_C(\Bar{u}, r_{\ell+1})|\cdot 1 \\
    &< \quad \frac{95|I|}{100} \cdot r_{\ell+1} + \frac{5|I|}{100} \cdot 1 \quad \\
    &< \quad \frac{|I|}{10}.
    \end{aligned}
    $$
    Here the second inequality follows from the fact that $r_{\ell+1} < 1$ and consequently maximizing $\left(|I| - |I\setminus B_C(\Bar{u}, r_{\ell+1})|\right)\cdot r_{\ell+1} + |I\setminus B_C(\Bar{u}, r_{\ell+1})|\cdot 1$ is equivalent to maximizing $|I\setminus B_C(\Bar{u}, r_{\ell+1})|$.
    But, Observation~\ref{obs:distance_defintions_match} implies the inequality $\sum_{v \in I} \min\{d(\Bar{u},v),1\} \geq \sum_{v \in I} d_{\Bar{u},v}$ which contradicts the constraint $\sum_{v \in I} d_{\Bar{u}, v} \geq \frac{|I|}{10}$ of LP~(\ref{lp:container_cover}). Thus, it follows that $|I \setminus B_C(\bar{u}, r_{\ell+1})| \geq \frac{5|I|}{100}$.
\end{proof}

\begin{lemma}\label{lem:good_layer_frac_separator}
    Let $Z \subseteq V(G)$ be a superset of $Z_0$ and let $\ell \in [\,\ell_{\max}\,]$. Let $C$ be a connected component of $G - Z$ with $\Bar{u} \in I \cap V(C)$. Let $A' := N[B_C(\Bar{u}, r_\ell)]\cap V(C)$ and $B' := V(C) \setminus B_C(\Bar{u}, r_{\ell+1})$. Then the assignment $\{x'_F\}_{F \in \mathcal{F}}$ defined by
    $$
    x'_F :=
    \begin{cases}
    \frac{x_F}{\epsilon} & \text{if } F \cap \delta_C(\Bar{u}, r_\ell) \neq \emptyset, \\
    0 & \text{otherwise.}
    \end{cases}
    $$
    is a fractional $(A', B')$-separator using $\mathcal{F}$ in $C$.
\end{lemma}
\begin{proof}

    Let $P = (v_1, \dots, v_p)$ be an $A'$--$B'$ path in $C$. Since $v_1\in A' = N[B_C(\Bar{u}, r_\ell)]\cap V(C)$, there exists $v_0\in B_C(\Bar{u}, r_\ell)$ with $v_1\in N[v_0]\cap V(C)$. By applying Observation~\ref{obs:triangle_inequality} to $\{\Bar{u},v_0,v_1\}$, we get $d(\Bar{u}, v_1) \leq d(\Bar{u}, v_0) + x_{v_1} \leq r_\ell + \epsilon$. Furthermore, $d(\Bar{u}, v_p) > r_\ell + 4\epsilon$ since $v_p\in B'$. 
    
    Let $P' := (v'_1, \dots, v'_q)$ be the contiguous subpath of $P$ with the minimum number of vertices such that $v'_1$ satisfies $d(\Bar{u}, v'_1) \leq r_\ell + \epsilon$, and $v'_q$ satisfies $d(\Bar{u}, v'_{q}) > r_\ell + 4\epsilon$. 
    Note that $P'$ is well-defined, since $P$ already satisfies the required properties.
    Moreover, $P'$ has at least 3 vertices; otherwise, applying Observation~\ref{obs:triangle_inequality} to $\{\Bar{u}, v'_1, v'_{q}\}$ yields $d(\Bar{u}, v'_q) \leq d(\Bar{u}, v'_{1}) + x_{v'_{q}} \leq r_\ell + 2\epsilon$, which contradicts the assumption that $d(\Bar{u}, v'_q) > r_\ell + 4\epsilon$. 
    Finally, by definition, we have $V(P')\setminus \{v'_1,v'_q\}\subseteq \delta_C(\bar{u}, r_\ell)$, since the existence of a vertex $v'_i$ in $V(P')\setminus \{v'_1,v'_q\}$ with $d(\bar{u},v'_i) \leq r_\ell+\epsilon$ (respectively $d(\bar{u},v'_i) > r_\ell+4\epsilon$) contradicts the minimality of $P'$ since the subpath of $P'$ from $v'_{i}$ to $v'_q$ (respectively $v'_1$ to $v'_{i}$) is a strictly smaller path that satisfies the requirements. 
    Hence:
    $$
    \sum_{\substack{F \in \mathcal{F}\\ F\cap P\neq \emptyset}} x'_F 
    % \quad = \sum_{\substack{F \in \mathcal{F} \\ F \cap \delta_C(\Bar{u}, r_\ell) \cap P \neq \emptyset}} x_F 
    \ \geq\  \sum_{\substack{F \in \mathcal{F} \\ F \cap \{v'_2, \dots, v'_{q-1}\} \neq \emptyset}} x'_F 
    \ \geq\  \frac{1}{\epsilon}\cdot d(v'_2,v'_{q-1})
    \ \geq\  \frac{1}{\epsilon}\cdot ((d(\Bar{u},v'_{q})-x_{v'_{q}})-(d(\Bar{u},v'_1)+x_{v'_2}))
    \ \geq\  1,
    $$ 
   where the second-to-last inequality follows from applying Observation~\ref{obs:triangle_inequality} to the triples $\{\Bar{u}, v'_2, v'_{q-1}\}$, $\{\Bar{u}, v'_{q-1}, v'_{q}\}$ and $\{\Bar{u}, v'_1, v'_{2}\}$. Thus $\{x'_F\}_{F \in \mathcal{F}}$ is a fractional $(A', B')$-separator using $\mathcal{F}$ in $C$.
\end{proof}

Let $X \subseteq V(G)$. We define $\mu(X) := \sum_{F \in \mathcal{F}} \alpha(F\cap X) \cdot x_F$. Now, we make some observations about the function $\mu$.

\begin{observation}\label{obs:mu_is_monotone}
    If $A$ and $B$ are subsets of $V(G)$ such that $A\subseteq B$, then $\mu(A) \leq \mu(B)$.
\end{observation}

\begin{observation}\label{obs:mu_is_additive}
   If $A$ and $B$ are disjoint and anti-complete subsets of $V(G)$, then $\mu(A \cup B) = \mu(A) + \mu(B)$.
\end{observation}

Observation~\ref{obs:mu_is_monotone} follows from the fact that for any $F \in \mathcal{F}$ and subsets $A \subseteq B \subseteq V(G)$, we have $\alpha(F\cap A) \leq \alpha(F\cap B)$. Similarly, Observation~\ref{obs:mu_is_additive} follows from the fact that if $A$ and $B$ are disjoint and anti-complete subsets of $V(G)$, then for every $F \in \mathcal{F}$ we have $\alpha(F\cap(A \cup B)) = \alpha(F\cap A) + \alpha(F\cap B)$.\\

\begin{lemma}\label{lem:good_layer_cost_and_gain}
    Let $Z \subseteq V(G)$ be a superset of $Z_0$ and let $C$ be the connected component of $G - Z$ such that $|I\cap V(C)| > \frac{95|I|}{100}$ with $\Bar{u} \in I \cap V(C)$. Then there exists $\ell \in [\,\ell_{\max}\,]$ such that $\mu(\delta_C(\Bar{u}, r_\ell)) \leq \mu(B_C(\Bar{u}, r_\ell))$.
\end{lemma}
\begin{proof}
    For this, we need the following claims:
    
    \begin{claim}\label{claim:inner_ball_has_non_zero_measure}
        $\mu(B_C(\Bar{u}, r_1)) \geq \epsilon$.
    \end{claim}
    \begin{proof}
        Consider a path $P$ in $C$ from $\Bar{u}$ to a vertex $w$ such that $w \in I\cap (V(C)\setminus B_C(\Bar{u}, r_{\ell_{\max} + 1}))$. Note that such a vertex exists by Lemma~\ref{lem:good_layer_many_outside}. Let $P' := (\Bar{u},v_1,\dots,v_p)$ be the subpath of $P$ such that the successor of $v_p$ (say $v_{p+1}$) is the first vertex in $P$ that does not belong to $B_C(\Bar{u}, r_1)$. $P'$ is well defined, since $w$ does not belong to $B_C(\Bar{u}, r_1)$ and is non-empty since $\Bar{u} \in B_C(\Bar{u}, r_1)$. Hence, by Observation~\ref{obs:triangle_inequality}, we get:
        $$
        \mu(B_C(\Bar{u}, r_1)) 
        \ \ \geq \ \ \sum_{\substack{F \in \mathcal{F} \\ F \cap B_C(\Bar{u}, r_1) \neq \emptyset}} x_F 
        \ \ \geq \ \ \sum_{\substack{F \in \mathcal{F} \\ F \cap P' \neq \emptyset}} x_F 
        \ \ \geq \ \ d(\Bar{u}, v_p) 
        \ \ \geq \ \ d(\Bar{u}, v_{p+1}) - x_{v_{p+1}}
        \ \ > \ \ \epsilon.
        $$ 
    \end{proof}
    
    \begin{claim}\label{claim:meause_of_ball_and_boundary_is_additive}
        $\mu(B_C(\Bar{u}, r_{\ell+1})) \geq \mu(B_C(\Bar{u}, r_\ell)) + \mu(\delta_C(\Bar{u}, r_\ell))$ for every $\ell \in [\,\ell_{\max}\,]$.
    \end{claim}
    \begin{proof}
        Let $v_1 \in B_C(\Bar{u}, r_\ell)$ and $v_2 \in \delta_C(\Bar{u}, r_\ell)$. We claim that $v_1$ and $v_2$ are non-adjacent. Indeed, if $v_1v_2\in E(G)$, then by applying Observation~\ref{obs:triangle_inequality} to $\{\Bar{u},v_1,v_2\}$, we get $d(\Bar{u}, v_2) \leq d(\Bar{u}, v_1) + x_{v_2} \leq r_\ell + \epsilon$, which contradicts $v_2\in \delta_C(\Bar{u}, r_\ell)$. Hence $v_1$ and $v_2$ are non-adjacent and consequently $B_C(\Bar{u}, r_\ell)$ and $\delta_C(\Bar{u}, r_\ell)$ are anti-complete in $G$. Since $B_C(\Bar{u}, r_\ell)$ and $\delta_C(\Bar{u}, r_\ell)$ are disjoint by definition, Observation~\ref{obs:mu_is_additive} implies that $\mu\left(\delta_C(\Bar{u}, r_\ell) \cup B_C(\Bar{u}, r_\ell)\right) = \mu\left(B_C(\Bar{u}, r_\ell)\right) + \mu\left(\delta_C(\Bar{u}, r_\ell)\right)$. Furthermore, by definition, $\delta_C(\Bar{u}, r_\ell) \cup B_C(\Bar{u}, r_\ell) \subseteq B_C(\Bar{u}, r_{\ell+1})$, and hence the claim follows from Observation~\ref{obs:mu_is_monotone}.
    \end{proof}

    Now, combining Claims~\ref{claim:inner_ball_has_non_zero_measure} and~\ref{claim:meause_of_ball_and_boundary_is_additive}, we conclude that if $\mu(\delta_C(\Bar{u}, r_\ell)) > \mu(B_C(\Bar{u}, r_\ell))$ for every $\ell \in [\,\ell_{\max}\,]$, then 
    $$
    \mu(B_C(\Bar{u}, r_{\ell_{\max}})) \quad
    > \quad 2^{\ell_{\max} - 1} \cdot \mu(B_C(\Bar{u}, r_1)) \quad
    \geq \quad \frac{2^{2 \cdot \log(a \cdot \lpopt + 4) + 11} }{1300\log(a \cdot \lpopt + 4)} \quad
    > \quad a \cdot\lpopt.
    $$
    Since this contradicts $\mu(V(G)) \leq a \cdot\lpopt$, we conclude that there exists some $\ell \in [\,\ell_{\max}\,]$ for which $\mu(\delta_C(\Bar{u}, r_\ell)) \leq \mu(B_C(\Bar{u}, r_\ell))$.
\end{proof}

\begin{lemma}\label{lem:cutoff_procedure}
     Let $Z \subseteq V(G)$ be a superset of $Z_0$ and let $C$ be the connected component of $G - Z$ such that $|I\cap V(C)| > \frac{95|I|}{100}$ with $\Bar{u} \in I \cap V(C)$. Then there exists a partition $A\cup S\cup B$ of $V(C)$ such that:
    \begin{enumerate}
        \item $S$ is an $A$--$B$ separator in $G-Z$.
        \item $B \subsetneq V(C)$
        \item $|A \cap I| \leq \frac{95|I|}{100}$.
        \item $\mathrm{fcov}_{\mathcal{F}}(S) \leq 15600 \cdot \log 2n \cdot a \cdot \log(a \cdot \lpopt + 4) \cdot \mu(C\setminus N[B])$.
    \end{enumerate}
\end{lemma}
\begin{proof}
    We apply Lemma~\ref{lem:good_layer_cost_and_gain} to the instance defined by $(Z, \Bar{u})$, to obtain some $\ell \in [\,\ell_{\max}\,]$ such that $\mu(\delta_C(\Bar{u}, r_\ell)) \leq \mu(B_C(\Bar{u}, r_\ell))$. Now, let $A' := N[B_C(\Bar{u}, r_\ell)]\cap V(C)$ and $B' := V(C) \setminus B_C(\Bar{u}, r_{\ell+1})$. By Lemma~\ref{lem:good_layer_frac_separator} the assignment $\{x'_F\}_{F \in \mathcal{F}}$ defined as,
    $$
    x'_F :=
    \begin{cases}
    \frac{x_F}{\epsilon} & \text{if } F \cap \delta_C(\Bar{u}, r_\ell) \neq \emptyset, \\
    0 & \text{otherwise.}
    \end{cases}
    $$
    is a fractional $(A',B')$-separator in $C$. Therefore, the tuple $(C, a, \mathcal{F}, A', B', \{x'_F\}_{F \in \mathcal{F}})$ satisfies all the requirements of Theorem~\ref{thm:A-B_separator}. Let $S$ be the $A'$--$B'$ separator in $C$ whose existence is guaranteed by it. Define $B$ to be the set of vertices in connected components of $C - S$ that have a non-empty intersection with $B'$, and let $A := V(C) \setminus (B \cup S)$. Observe that, from definitions it follows that $S$ is an $A$--$B'$ separator and an $A$--$B$ separator in $C$ and consequently in $G-Z$. So, we have $|A \cap I| \leq |I| - |B' \cap I|$. But by Lemma~\ref{lem:good_layer_many_outside}, $|B' \cap I|$ is at least $\frac{5|I|}{100}$, implying that $|A \cap I| \leq \frac{95|I|}{100}$. 
    
    We claim that $B_C(\bar{u}, r_\ell) \cap N[B] = \emptyset$. Suppose not, and let $u \in B_C(\bar{u}, r_\ell) \cap N[B]$. Then there exists $v \in B$ such that $v \in N[u]$ which implies that $v \in A' \cap B$. However, this contradicts the fact that $S$ is an $A'$--$B$ separator in $C$. Hence, we conclude that $B_C(\bar{u}, r_\ell) \subseteq V(C) \setminus N[B]$, which, in particular, implies that $B \subsetneq V(C)$.
    
    Now, we bound the fractional cover number of $S$ using $\mathcal{F}$ as follows:
    \begin{align*}
        \mathrm{fcov}_{\mathcal{F}}(S) \quad
        &\leq \quad 12 \cdot \log 2n \cdot a \cdot \sum_{F \in \mathcal{F}} x'_F \\
        &= \quad 12 \cdot \log 2n \cdot a \cdot \sum_{\substack{F \in \mathcal{F} \\ F \cap \delta_C(\Bar{u}, r_\ell) \neq \emptyset}} \frac{x_F}{\epsilon} \\
        &\leq \quad 12 \cdot \log 2n \cdot a \cdot 1300 \cdot \log(a \cdot \lpopt + 4) \cdot \mu(\delta_C(\Bar{u}, r_\ell)) \\
        &\leq \quad 15600 \cdot \log 2n \cdot a \cdot \log(a \cdot \lpopt + 4) \cdot \mu(V(C)\setminus N[B]).
    \end{align*}
    where the last inequality follows from the fact that $\mu(\delta_C(\Bar{u}, r_\ell)) \leq \mu(B_C(\Bar{u}, r_\ell))$ and by applying Observation~\ref{obs:mu_is_monotone} to the sets $B_C(\Bar{u}, r_\ell)$ and $V(C) \setminus N[B]$.

\end{proof}

Now we are ready to prove Theorem~\ref{thm:balanced_separator_rounding}.

\LPoptVSintegral*
\begin{proof}
    % We prove the following claim:
    % \begin{claim}
    %     Let $Z\subseteq V(G)$ be a superset of $Z_0$, and let $C$ be the connected component of $G - Z$ such that $|C \cap I| > \frac{95|I|}{100}$, or let $C = G[\emptyset]$ if no such component exists. Then there exists an $(I, \frac{95}{100})$-balanced separator $S$ in $G$ such that $\mathrm{fcov}_{\mathcal{F}}(S) \leq \mathrm{fcov}_{\mathcal{F}}(Z) + 15600 \cdot \log 2n \cdot a \cdot  \log(a \cdot \lpopt + 4) \cdot \mu(C)$.
    % \end{claim}
    We prove the following claim:\\[-4pt]
    
    \noindent\textbf{Claim~6.2.1.}\label{claim:balanced_sep_cutoff_step}
        \emph{Let $Z\subseteq V(G)$ be a superset of $Z_0$, and let $C$ be the connected component of $G - Z$ such that $|I\cap V(C)| > \frac{95|I|}{100}$, or let $C = G[\emptyset]$ if no such component exists. Then there exists an $(I, \frac{95}{100})$-balanced separator $S$ in $G$ such that $\mathrm{fcov}_{\mathcal{F}}(S) \leq \mathrm{fcov}_{\mathcal{F}}(Z) + 15600 \cdot \log 2n \cdot a \cdot \log(a \cdot \lpopt + 4) \cdot \mu(V(C))$.}
    \begin{proof}
        We proceed by induction on $|V(C)|$. Firstly, if $C = G[\emptyset]$, then we can let $S = Z$ and the lemma holds true. Otherwise, let $A'\cup S'\cup B'$ be the partition of $V(C)$ whose existence is implied by Lemma~\ref{lem:cutoff_procedure} and consider the graph $G-(Z\cup S')$. If there exists a component $C'$ of $G - (Z \cup S')$ such that $|V(C') \cap I| > \frac{95|I|}{100}$, then we observe that $V(C') \subseteq B'$, since $S'$ is an $A'$--$B'$ separator in $G - Z$ and $|A' \cap I| \leq \frac{95|I|}{100}$. Hence, $V(C') \subsetneq V(C)$, as $B' \subsetneq V(C)$. Otherwise, if no such component exists, we have $V(C') = \emptyset$ and consequently $V(C')\subsetneq V(C)$. In either case, by applying the inductive hypothesis to the pair $(G, Z \cup S')$, we obtain an $(I, \frac{95}{100})$-balanced separator $S$ in $G$ such that:
        \begin{align*}
            \mathrm{fcov}_{\mathcal{F}}(S) \quad
            &\leq \quad \mathrm{fcov}_{\mathcal{F}}(Z \cup S')\ +\ 15600 \cdot \log 2n \cdot a \cdot  \log(a \cdot \lpopt + 4) \cdot \mu(V(C')) \\
            &\leq \quad \mathrm{fcov}_{\mathcal{F}}(Z)\ +\ \mathrm{fcov}_{\mathcal{F}}(S')\ +\ 15600 \cdot \log 2n \cdot a \cdot  \log(a \cdot \lpopt + 4) \cdot \mu(V(C')) \\
            &\leq \quad \mathrm{fcov}_{\mathcal{F}}(Z)\ +\ 15600 \cdot \log 2n \cdot a \cdot  \log(a \cdot \lpopt + 4) \cdot \left(\mu(V(C)\setminus N[B'])\ +\ \mu(V(C'))\right) \\
            &\leq \quad \mathrm{fcov}_{\mathcal{F}}(Z)\ +\ 15600 \cdot \log 2n \cdot a \cdot  \log(a \cdot \lpopt + 4) \cdot \mu(V(C)).
        \end{align*}
        Here, the last inequality follows from applying Observations~\ref{obs:mu_is_additive} and~\ref{obs:mu_is_monotone}, using the facts that $V(C') \subseteq B'$ and that $V(C) \setminus N[B']$ and $B'$ are disjoint and anti-complete in $G$.
    \end{proof}
    Now, we apply Claim~\hyperref[claim:balanced_sep_cutoff_step]{6.2.1} to the pair $(G, Z_0)$ and obtain an $(I, \frac{95}{100})$-balanced separator $S$ in $G$ satisfying $\mathrm{fcov}_{\mathcal{F}}(S) \leq \mathrm{fcov}_{\mathcal{F}}(Z_0) + 15600 \cdot \log 2n \cdot a \cdot  \log(a \cdot \lpopt + 4) \cdot \mu(V(C))$. Furthermore, we observe that, $\left\{\frac{x_F}{\epsilon}\right\}_{F \in \mathcal{F}}$ is a fractional cover for $Z_0$, and consequently $\mathrm{fcov}_{\mathcal{F}}(Z_0) \leq \frac{1}{\epsilon} \lpopt$. Thus, 
    \begin{align*}
        \mathrm{fcov}_{\mathcal{F}}(S) \quad
        &\leq \quad \mathrm{fcov}_{\mathcal{F}}(Z_0)\ +\ 15600 \cdot \log 2n \cdot a \cdot  \log(a \cdot \lpopt + 4) \cdot \mu(V(C)) \\
        &\leq \quad 1300 \cdot  \log(a \cdot \lpopt + 4) \cdot \lpopt\ +\ 15600 \cdot \log 2n \cdot a \cdot  \log(a \cdot \lpopt + 4) \cdot \mu(V(G)) \\
        &\leq \quad 17000 \cdot \log 2n \cdot a^2 \cdot  \log(a \cdot \lpopt + 4) \cdot \lpopt
    \end{align*}
    which concludes the proof of Theorem~\ref{thm:balanced_separator_rounding}.
\end{proof}
    
\section{Path Packing via the Dual Balanced Separator LP}\label{sec:dual_of_balanced_separator}
% \section{Obstructions to Balanced Separators with Small Independence Number}\label{sec:dual_of_balanced_separator}
    Let $G$ be a graph, $I \subseteq V(G)$ an independent set, and $\mathcal{F}$ a family of vertex subsets. Let $\mathcal{P}_{u,v}$ denote the set of all induced paths from $u$ to $v$ in $G$. We now describe the dual of the balanced separator linear program corresponding to the instance $(G, I, \mathcal{F})$. It uses non-negative real variables $\rho_u$, $\eta_{u,v}$, and $\gamma_{u,v,P}$, defined for every $u,v \in I$ and path $P \in \mathcal{P}_{u,v}$.
\begin{align}\label{lp:balanced_path_packing}
    \text{Maximize :}&\quad \frac{|I|}{10}\sum\limits_{u\in I} \rho_u - \sum\limits_{u,v \in I}\eta_{u,v}\\
    \nonumber \text{subject to :}&\quad \rho_u - \eta_{u,v} - \sum\limits_{P\in\mathcal{P}_{u,v}}\gamma_{u,v,P} \leq 0 && \forall u,v \in I\\
    \nonumber &\quad \sum\limits_{\substack{u,v \in I\\ P\in \mathcal{P}_{u,v}\\ P\cap F\neq \emptyset}} \gamma_{u,v,P} \leq 1 && \forall F\in \mathcal{F}
\end{align}
It is easy to verify that the above is indeed the dual of LP~(\ref{lp:container_cover}). Having established this, we are now ready to state the main theorem proved in this section.

\begin{restatable}{theorem}{BalancedSeparatorDualRounding}\label{thm:balanced_separator_dual_rounding}
    For every graph $G$, independent set $I \subseteq V(G)$, positive integers $a$, $b$ and $\ell$ such that $\ell\ \geq\ 7 \cdot (\,\lpopt \cdot \log (4|\mathcal{F}|)\, +\, |I|\,)$, where $\mathcal{F}$ is a $(b,a)$-container family in $G$, there exists a subgraph $H \subseteq G$ with $I \subseteq V(H)$ that satisfies the following properties:
    \begin{itemize}
        \item Every induced subgraph of $H$ with independence number at most $b$, has at most $\frac{3 \cdot b \cdot \ell}{\lpopt}$ vertices.
        % \item Every induced subgraph of $H$ with independence number at most $b$, has at most $\frac{3 \cdot a \cdot \ell}{\lpopt}$ vertices.
        \item For every $(I,\frac{1}{2})$-balanced separator $S$ in $H$, we have that $\mathrm{cov}_{\mathcal{F}}(S) \geq \lpopt$.
    \end{itemize}
    whenever $\lpopt$, which denotes the optimal value of the balanced separator linear program corresponding to $(G, I, \mathcal{F})$, is positive.
\end{restatable}

For the remainder of this section, we fix a graph $G$, an independent set $I \subseteq V(G)$, and a $(b,a)$-container family $\mathcal{F}$ in $G$. We fix an optimal solution to LP~(\ref{lp:balanced_path_packing}), the dual of the balanced separator linear program corresponding to $(G, I, \mathcal{F})$. Note that such a solution exists and is bounded. Indeed, since every $(b,a)$-container family satisfies the conditions of Observation~\ref{obs:lp_is_feasible}, the balanced separator linear program for $(G, I, \mathcal{F})$ is feasible and admits a bounded solution; consequently by strong duality~\cite{doi:10.1137/1025101}, the same holds for its dual. Let $\rho_u$, $\eta_{u,v}$ and $\gamma_{u,v,P}$, for every $u,v \in I$ and $P \in \mathcal{P}_{u,v}$, denote the values assigned to the corresponding variables in this solution. Let $\lpopt$ be the value of the objective function in this solution. For every $u, v \in I$, define $\gamma_{u,v} := \sum_{P \in \mathcal{P}_{u,v}} \gamma_{u,v,P}$, and let $\rho := \sum_{u \in I} \rho_u$.

\begin{lemma}\label{lem:properties_of_dual_lp}
    The optimal solution has the following properties:
    \begin{enumerate}
        \item $\sum_{v\in I}\eta_{u,v} \leq \frac{|I|}{10}\rho_u$ for every $u\in I$.
        \item $10\,\lpopt \leq \rho|I|$.
    \end{enumerate}
\end{lemma}

\begin{proof}
    Let $u\in I$. Given any feasible solution to the linear program, we can always obtain another solution by setting $\rho_{u} = 0$ and $\eta_{u,v} = 0$ for every $v\in I$. This preserves feasibility, since the only constraints involving these variables are of the form $\rho_{u} - \eta_{u,v} - \gamma_{u,v} \leq 0$, which remain satisfied even in the new assignment. Furthermore, the value of the objective function decreases by $\frac{|I|}{10}\rho_{u} - \sum_{v\in I}\eta_{u,v}$. But, if we start with the optimal solution, this modification should not strictly increase the objective function value. Therefore we get $\sum_{v\in I}\eta_{u,v} \leq \frac{|I|}{10}\rho_{u}$ for every $u\in I$. Furthermore,
    $$
    10\,\lpopt \quad
    = \quad 10\left(\frac{|I|}{10} \sum_{u \in I} \rho_u - \sum_{u,v \in I} \eta_{u,v}\right) \quad
    \leq \quad 10\left(\frac{|I|}{10} \sum_{u \in I} \rho_u\right) \quad
    = \quad \rho|I|.
    $$
\end{proof}
    
We define $\mathcal{P} = \bigcup_{u,v \in I} \mathcal{P}_{u,v}$, and introduce a probability distribution $\mathcal{D}$ over the domain $I\times I\times (\mathcal{P} \cup \{\emptyset\})$. The distribution $\mathcal{D}$ is defined using the following random process, which has three steps: (Step~1) Select a vertex $u \in I$ with probability $\rho_u / \rho$. (Step~2) Select a vertex $v \in I$ uniformly at random.  (Step~3) With probability $\eta_{u,v} / (\eta_{u,v} + \gamma_{u,v})$, output the triple $(u, v, \emptyset)$. Otherwise, choose a path $P \in \mathcal{P}_{u,v}$ with probability $\gamma_{u,v,P} / \gamma_{u,v}$ and output the triple $(u, v, P)$. 

We now argue that the process is well defined.
First, observe that $\lpopt > 0$ by assumption and $10\cdot\lpopt \leq \rho |I|$ by Lemma~\ref{lem:properties_of_dual_lp}, together imply $\rho > 0$. Hence $\{\, \rho_u / \rho \,\}_{u \in I}$ forms a probability distribution: each $\rho_u$ is non-negative and $\sum_{u \in I} \rho_u = \rho$. Therefore, Step~1 is well defined.
We next justify Step~3, as Step~2 is immediate.
Let $u$ be a vertex selected in Step~1. Then $\rho_u > 0$, which implies that $\eta_{u,v} + \gamma_{u,v} > 0$ for every $v \in I$, and in particular for the vertex $v$ selected in Step~2; otherwise, the constraint $\rho_u - \eta_{u,v} - \gamma_{u,v} \leq 0$ of LP~(\ref{lp:balanced_path_packing}) would be violated. If $\gamma_{u,v} = 0$, then necessarily $\eta_{u,v} / (\eta_{u,v} + \gamma_{u,v}) = 1$, and the process outputs $(u,v,\emptyset)$. Thus, the second case of Step~3 occurs only when $\gamma_{u,v} > 0$. In this case, $\{\, \gamma_{u,v,P} / \gamma_{u,v} \,\}_{P \in \mathcal{P}_{u,v}}$ forms a probability distribution, since each $\gamma_{u,v,P}$ is non-negative and $\sum_{P \in \mathcal{P}_{u,v}} \gamma_{u,v,P} = \gamma_{u,v}$.
Consequently, Step~3, and hence the entire process, is well defined.

\begin{lemma}\label{lem:probability_bounds_1}
    $\mathcal{D}$ satisfies $\displaystyle \mathop{\mathbb{P}}\limits_{(u,v,P)\sim\mathcal{D}}\left[P = \emptyset\right] \leq \frac{1}{10}$.
\end{lemma}
\begin{proof}
    Using Lemma~\ref{lem:properties_of_dual_lp}, together with the fact that the constraint $\rho_u - \eta_{u,v} - \gamma_{u,v} \leq 0$ holds for every $u, v \in I$, we obtain:
    \begin{align*}
        \displaystyle \mathop{\mathbb{P}}\limits_{(u,v,P)\sim\mathcal{D}}\left[P = \emptyset\right] \ 
        &= \ \sum_{\substack{u' \in I \\ \rho_{u'} > 0}} \sum_{v' \in I} \displaystyle \mathop{\mathbb{P}}\limits_{(u,v,P)\sim\mathcal{D}}\left[P = \emptyset\,\mid\, u = u',\, v = v'\right] \cdot \mathop{\mathbb{P}}\limits_{(u,v,P)\sim\mathcal{D}}\left[v = v'\right] \cdot \mathop{\mathbb{P}}\limits_{(u,v,P)\sim\mathcal{D}}\left[u = u'\right] \\
        % &= \ \sum_{u' \in I} \sum_{v' \in I} \frac{\eta_{u',v'}}{\eta_{u',v'} + \gamma_{u',v'}} \cdot \frac{1}{|I|} \cdot \frac{\rho_{u'}}{\rho} \\
        % &\leq \ \frac{1}{|I|} \sum_{u' \in I} \frac{\rho_{u'}}{\rho} \sum_{v' \in I} \frac{\eta_{u',v'}}{\rho_{u'}} \quad
        &= \ \frac{1}{|I|} \sum_{\substack{u' \in I \\ \rho_{u'} > 0}} \frac{\rho_{u'}}{\rho} \sum_{v' \in I} \frac{\eta_{u',v'}}{\eta_{u',v'} + \gamma_{u',v'}} \quad
        \leq \quad \frac{1}{|I|} \sum_{\substack{u' \in I \\ \rho_{u'} > 0}} \frac{1}{\rho} \cdot \frac{|I|}{10} \rho_{u'} \quad
        = \quad \frac{1}{10}.
    \end{align*}
\end{proof}

\begin{lemma}\label{lem:probability_bounds_2}
    If $F\in \mathcal{F}$, then $\displaystyle \mathop{\mathbb{P}}\limits_{\substack{(u,v,P)\sim\mathcal{D}}}\left[P \cap F \neq \emptyset\right] \leq \frac{1}{10\cdot\lpopt}$.
\end{lemma}
\begin{proof}
    Let $u',v'\in I$ and $P'\in\mathcal{P}_{u',v'}$. Observe that if $\rho_{u'} = 0$ or $\gamma_{u',v'} = 0$, then it follows from the definition of the process that,
    $$
    \mathop{\mathbb{P}}\limits_{\substack{(u,v,P)\sim\mathcal{D}}}\left[(u,v,P) = (u',v',P')\right] = 0.
    $$
    Hence, let us consider the case where $\rho_{u'} > 0$ and $\gamma_{u',v'} > 0$. Using the fact that the constraint $\rho_u - \eta_{u,v} - \gamma_{u,v} \leq 0$ holds for every $u, v \in I$, we obtain:
    \begin{align*}
        \displaystyle \mathop{\mathbb{P}}\limits_{\substack{(u,v,P)\sim\mathcal{D}}}[(u,v,P) = (u',v',P')] \ 
        % &= \ \mathbb{P}_{\substack{(u,v,P)\sim\mathcal{D}}}\left[P = P' \mid u = u'\right] \cdot \mathbb{P}_{\substack{(u,v,P)\sim\mathcal{D}}}\left[u = u'\right] \\
        &= \ \displaystyle \mathop{\mathbb{P}}\limits_{\substack{(u,v,P)\sim\mathcal{D}}}\left[P = P' \mid v = v', u = u'\right] \cdot \mathop{\mathbb{P}}\limits_{\substack{(u,v,P)\sim\mathcal{D}}}\left[v = v'\right] \cdot \mathop{\mathbb{P}}\limits_{\substack{(u,v,P)\sim\mathcal{D}}}\left[u = u'\right] \\
        &= \ \frac{\gamma_{u',v',P'}}{\gamma_{u',v'}} \cdot \frac{\gamma_{u',v'}}{\gamma_{u',v'} + \eta_{u',v'}} \cdot \frac{1}{|I|} \cdot \frac{\rho_{u'}}{\rho} \quad 
        % &= \ \frac{\gamma_{u',v',P'}}{\rho |I|} \cdot \frac{\rho_{u'}}{\gamma_{u',v'} + \eta_{u',v'}} \\
        \leq \quad \frac{\gamma_{u',v',P'}}{\rho |I|}.
    \end{align*}
    Now, we fix $F \in \mathcal{F}$ and consider the probability that the sampled set $P$ has a non-empty intersection with $F$.
    \begin{align*}
        \displaystyle \mathop{\mathbb{P}}\limits_{\substack{(u,v,P)\sim\mathcal{D}}}\left[\,P \cap F \neq \emptyset\,\right] \quad
        &= \quad \sum_{u',v'\in I} \sum_{\substack{P' \in \mathcal{P}_{u',v'} \\ P' \cap F \neq \emptyset}} 
        \displaystyle \mathop{\mathbb{P}}\limits_{\substack{(u,v,P)\sim\mathcal{D}}}\left[\, (u,v,P) = (u',v',P')\, \right] \\
        &\leq \quad \frac{1}{\rho |I|}\sum_{\substack{u',v'\in I\\ P' \in \mathcal{P}_{u',v'} \\\rho_{u'} > 0\\ \gamma_{u',v'} > 0 \\ P' \cap F \neq \emptyset}} \gamma_{u',v',P'} \quad 
        \leq \quad \frac{1}{\rho |I|} \quad
        \leq \quad \frac{1}{10 \cdot \lpopt}.
    \end{align*}
    Here, the last inequality follows from Lemma~\ref{lem:properties_of_dual_lp}, and the second-to-last follows from the fact that 
    %the variables $\{\,\gamma_{u,v,P}\, \mid\, u,v \in I,\; P \in \mathcal{P}_{u,v}\,\}$ form part of a feasible solution to LP~\ref{lp:balanced_path_packing}, which has 
    the constraint, $\sum_{\substack{u,v \in I\\ P\in \mathcal{P}_{u,v}\\ P\cap F\neq \emptyset}} \gamma_{u,v,P} \leq 1$, holds for every $F \in \mathcal{F}$.
\end{proof}

Finally, we require the following lemma, which relates balanced separators of the set $A$ to $A_1$--$A_2$ separators for a specific partition $A_1$, $A_2$ of $A$.

\begin{lemma}\label{lem:separation_vs_separators}
    Let $G$ be a graph and $A\subseteq V(G)$ be a vertex subset. If $G$ has an $(A,\frac{1}{2})$-balanced separator $S$, then there exists a partition $A_1$, $A_2$ of $A$ such that $\max\{|A_1|, |A_2|\} \leq \frac{2|A|}{3}$ and $S$ is an $A_1$--$A_2$ separator.
\end{lemma}
\begin{proof}
    If $S\cap A$ contains at least $\frac{|A|}{2}$ elements, then selecting any subset $A_1 \subseteq S\cap A$ of size $\frac{|A|}{2}$ and setting $A_2 := A \setminus A_1$ suffices to prove the lemma. Otherwise, let $C_1, \dots, C_r$ be the connected components of $G-S$ that have a non-empty intersection with $A$. Let $A'_i := A \cap V(C_i)$ for every $i \in [r]$, and let $A'_{r+1} := A \cap S$. By relabeling the sets if necessary, assume that $|A'_i| \geq |A'_{i+1}|$ for all $i \in [r]$. Let $q$ be the smallest integer in $[r+1]$ such that $\sum_{i=1}^{q}|A'_i| \geq \frac{|A|}{3}$. We claim that $\sum_{i=1}^{q}|A'_i| \leq \frac{2|A|}{3}$. If $q = 1$, the claim holds trivially, as $S$ is an $(A,\frac{1}{2})$-balanced separator and consequently $|A'_1|\leq \frac{|A|}{2}$. Otherwise, observe that
    $$
    \sum_{i=1}^{q}|A'_i| \quad
    = \quad |A'_{q}|\ +\ \sum_{i=1}^{q-1}|A'_i| \quad
    \leq \quad |A'_{q-1}|\ +\ \sum_{i=1}^{q-1}|A'_i| \quad
    \leq \quad 2\sum_{i=1}^{q-1}|A'_i| \quad
    \leq \quad \frac{2|A|}{3}.
    $$
    Define $A_1 := \cup_{i=1}^{q} A'_i$ and $A_2 := A\setminus A_1$. Since the sets in $\{A'_i\}_{i\in [r+1]}$ are pairwise disjoint, we have that $|A_1| = \sum_{i=1}^{q}|A'_i|$. Hence, $|A_1| \leq \frac{2|A|}{3}$ and since $|A_1| \geq \frac{|A|}{3}$, we have $|A_2| \leq \frac{2|A|}{3}$. Furthermore, it follows from our construction of $A_1$, $A_2$ that $S$ is an $A_1$--$A_2$ separator, which concludes the proof of the lemma.
\end{proof}

Now we are ready to prove Theorem~\ref{thm:balanced_separator_dual_rounding}.

\BalancedSeparatorDualRounding*
\begin{proof}
    Let $\{(u_i, v_i, P_i)\}_{i\in[\,\ell\,]}$ be $\ell$ independent samples drawn from the distribution $\mathcal{D}$ over $I\times I\times (\mathcal{P} \cup \{\emptyset\})$. Let $H$ be the subgraph of $G$ induced by the union of all sampled sets and $I$, that is, $H = G[I\cup\bigcup_{i \in [\ell]} V(P_i)]$. We claim that $H$ has the desired properties with good probability.

    \begin{claim}\label{claim:small_clique}
        With probability at least $\frac{3}{4}$, the following holds: every induced subgraph of $H$ with independence number at most $b$, has at most $\frac{3 \cdot b \cdot \ell}{\lpopt}$ vertices.
    \end{claim}
    % \begin{claim}\label{claim:small_clique}
    %     With probability at least $\frac{3}{4}$, the following holds: every induced subgraph of $H$ with independence number at most $b$, has at most $\frac{3 \cdot a \cdot \ell}{\lpopt}$ vertices.
    % \end{claim}
    \begin{proof}
        Let $F \in \mathcal{F}$, and let $\chi_F$ denote the number of sets $P_i$, out of the $\ell$ samples, that intersect $F$. We define $\chi_{\mathcal{F}} := \max\{\chi_F\, |\, F\in\mathcal{F}\}$ and show that with probability at least $\frac{3}{4}$, the event $\chi_{\mathcal{F}} < 6 \cdot \frac{\ell}{10\,\lpopt}$ occurs. 
        
        Let $F\in\mathcal{F}$. Since the probability that a fixed set $P_i$ intersects $F$ is at most $\frac{1}{10\, \lpopt}$, it follows from linearity of expectation that the expected value of $\chi_F$ is at most $\frac{\ell}{10\, \lpopt}$. Applying union bound over all $F\in\mathcal{F}$, the Chernoff bound from Proposition~\ref{thm:chernoff} to $\chi_F$ for every $F\in\mathcal{F}$, and using the lower bound on $\ell$, we get: 
        $$
        \mathbb{P}\left[\chi_{\mathcal{F}} < 6 \cdot \frac{\ell}{10\,\lpopt}\right] \quad
        \geq \quad 1\ -\ \sum_{F\in\mathcal{F}}\mathbb{P}\left[\chi_F \geq 6 \cdot \frac{\ell}{10\,\lpopt}\right] \quad
        \geq \quad 1\ -\ |\mathcal{F}| \cdot 2^{-6 \cdot \frac{\ell}{10\,\lpopt}} \quad
        \geq \quad \frac{3}{4}.
        $$
        Suppose that in our sampled subgraph $H$, the event $\chi_{\mathcal{F}} < 6 \cdot \frac{\ell}{10\,\lpopt}$ occurs. Let $H'$ be an induced subgraph of $H$ with independence number at most $b$. Since the paths in $\mathcal{P}$ are induced and the independence number of $H'$ is at most $b$, we have $|V(P_i) \cap H'| \leq 2b$, for every $i \in [\ell]$. As $H'$ is an induced subgraph of $G$, and since $\mathcal{F}$ is a $(b,a)$-container family, there exists some $F \in \mathcal{F}$ such that $V(H')\subseteq F$. Therefore, the number of sets in $P_1,\,\dots\,, P_{\ell}$ that intersect $H'$ is at most $\chi_{\mathcal{F}}$. Furthermore, since $I$ is an independent set in $G$, we have $|I\cap H'| \leq b$. Thus, $|V(H')| \leq |I\cap V(H')| + \sum_{i\in[\ell]}|V(P_i) \cap H'| < b + 2b \cdot \chi_F$. Therefore, with probability at least $\frac{3}{4}$, we have that every induced subgraph of $H$ with independence number at most $b$, has at most $\frac{3 \cdot b \cdot \ell}{\lpopt}$ vertices.
        
        % Suppose that in our sampled subgraph $H$, the good event $\chi_{\mathcal{F}}(H) < 6 \cdot \frac{\ell}{10,\lpopt}$ occurs. Now, observe that since the paths in $\mathcal{P}$ are induced, and the independence number of every $F \in \mathcal{F}$ is at most $a$, it follows that for every $i \in [\ell]$ and every $F \in \mathcal{F}$, we have $|V(P_i) \cap F| \leq 2a$. Furthermore, since $I$ is an independent set in $G$, we have $|I\cap F| \leq a$ for every $F \in \mathcal{F}$.
        
        % Let $H'$ be an induced subgraph of $H$ with independence number at most $b$. Since $H'$ is also an induced subgraph of $G$ with the same properties, and since $\mathcal{F}$ is a $(b,a)$-container family, there exists some $F \in \mathcal{F}$ such that $V(H')\subseteq F$. It follows that $|H'| \leq |F \cap V(H)| \leq a\, +\, 2a \cdot \chi_F$. Therefore, with probability at least $\frac{3}{4}$, we have that every induced subgraph of $H$ with independence number at most $b$, has at most $\frac{3 \cdot a \cdot \ell}{\lpopt}$ vertices.
    \end{proof}

    \begin{claim}\label{claim:no_balanced_separator}
        With probability at least $\frac{3}{4}$, the following holds: for every subfamily $\mathcal{F}' \subseteq \mathcal{F}$ of size at most $\lpopt$, and for every partition $I_1, I_2$ of $I$ such that $\max\{|I_1|, |I_2|\} \leq \frac{2|I|}{3}$, there exists an $I_1$--$I_2$ path in $H$ that does not intersect any set in $\mathcal{F}'$. 
    \end{claim}

    \begin{proof}
        We define an {\em eligible triple} to be a triple $(\mathcal{F}', I_1, I_2)$ such that $\mathcal{F}' \subseteq \mathcal{F}$ is a subfamily of size at most $\lpopt$, and $I_1$, $I_2$ is a partition of $I$ satisfying $\max\{|I_1|, |I_2|\} \leq \frac{2|I|}{3}$. An eligible triple $(\mathcal{F}', I_1, I_2)$ is {\em bad} if every $I_1$--$I_2$ path in $H$ intersects at least one set in $\mathcal{F}'$. Let $\chi$ be the event that the conclusion of the claim is false, namely that there exists an eligible bad triple. To prove the claim it suffices to show $\mathbb{P}[\chi] \leq \frac{1}{4}$.
        
        For every eligible triple  $(\mathcal{F}', I_1, I_2)$, we denote by $\chi({\cal F}', I_1, I_2)$ the event that this triple is bad. We will prove that for every eligible triple $(\mathcal{F}', I_1, I_2)$ we have $\mathbb{P}[\chi({\cal F}', I_1, I_2)] \leq (\frac{9}{10})^\ell$. Then $\mathbb{P}[\chi] \leq \frac{1}{4}$ follows by a simple union bound over all eligible triples. More concretely we have the following.
        $$
            \mathbb{P}\left[\chi\right] \quad \leq \quad \left(\frac{9}{10}\right)^{\ell} \cdot 2^{|I|} \cdot \binom{|\mathcal{F}|}{|\mathcal{F}'|} \quad
            \leq \quad 2^{|I|\ +\ |\mathcal{F}'| \log |\mathcal{F}|\ -\ \ell \log\left(\frac{10}{9}\right)} \quad
            \leq \quad \frac{1}{4}.
        $$
        
        To prove that $\mathbb{P}[\chi({\cal F}', I_1, I_2)] \leq (\frac{9}{10})^\ell$ we observe that the event $\chi({\cal F}', I_1, I_2)$ does not occur if there exists an $i \in [\ell]$ such that $P_i$ is an $I_1$--$I_2$ path in $H$ that is disjoint from every set in $\mathcal{F}'$. For every eligible triple $(\mathcal{F}', I_1, I_2)$ and every $i \in [\ell]$ we define the event $\overline{\chi}({\cal F}', I_1, I_2, i)$ that $P_i$ is an $I_1$--$I_2$ path in $H$ that is disjoint from every set in $\mathcal{F}'$. For every eligible triple $({\cal F}', I_1, I_2)$ the events in $\{\overline{\chi}({\cal F}', I_1, I_2, i)\}_{i \in [\ell]}$ are independent. Thus, to prove  $\mathbb{P}[\chi({\cal F}', I_1, I_2)] \leq (\frac{9}{10})^\ell$ it suffices to show that $\mathbb{P}[\overline{\chi}({\cal F}', I_1, I_2, i)] \geq \frac{1}{10}$ for every $i \in [\ell]$.
        
        To lower bound $\mathbb{P}[\overline{\chi}({\cal F}', I_1, I_2, i)]$ we observe that $(u_i, v_i, P_i)$ is sampled according to the distribution ${\cal D}$. The event $\overline{\chi}({\cal F}', I_1, I_2, i)$ occurs unless $|\{u_i, v_i\} \cap I_1| \neq 1$, or $P_i = \emptyset$, or there exists $F \in {\cal F}'$ such that $P_i \cap F \neq \emptyset$. Since $v_i$ is sampled uniformly from $I$ we have that  $\mathbb{P}[v_i \in I_1 | u_i \in I_1] \leq \frac{2}{3}$ and $\mathbb{P}[v_i \in I_2 | u_i \in I_2] \leq \frac{2}{3}$. The law of conditional probability applied to the event $u_i \in I_1$ now yields $\mathbb{P}[|\{u_i, v_i\} \cap I_1| \neq 1] \leq \frac{2}{3}$. By Lemma~\ref{lem:probability_bounds_1} we have that $\mathbb{P}[P_i = \emptyset] \leq \frac{1}{10}$. For every $F \in {\cal F}'$, by Lemma~\ref{lem:probability_bounds_2} we have that $\mathbb{P}[P_i \cap F \neq \emptyset] \leq \frac{1}{10 \cdot \lpopt}$. A union bound over all $F \in {\cal F}'$ yields that the probability that there exists $F \in {\cal F}'$ such that $P_i \cap F \neq \emptyset$ is at most $\frac{|{\cal F}'|}{10 \cdot \lpopt} \leq \frac{1}{10}$. We conclude that $\mathbb{P}[\overline{\chi}({\cal F}', I_1, I_2, i)] \geq 1 - \frac{2}{3} - \frac{1}{10} - \frac{1}{10} \geq \frac{1}{10}$, giving the desired lower bound on $\mathbb{P}[\overline{\chi}({\cal F}', I_1, I_2, i)]$, and completing the proof of the claim. 
    \end{proof}

Since the probability that the sampled subgraph satisfies both the properties stated in Claim~\ref{claim:small_clique} and Claim~\ref{claim:no_balanced_separator} is at least $\frac{1}{2}$, there exists an induced subgraph $H \subseteq G$ that satisfies both. Let $H$ be such an induced subgraph. By Claim~\ref{claim:small_clique} every induced subgraph of $H$ with independence number at most $b$, has at most $\frac{3 \cdot b \cdot \ell}{\lpopt}$ vertices. Furthermore, combining Lemma~\ref{lem:separation_vs_separators} with Claim~\ref{claim:no_balanced_separator}, we conclude that $H$ has no $(I,\frac{1}{2})$-balanced separator $S$ with $\mathrm{cov}_{\mathcal{F}}(S) \leq \lpopt$, thereby completing the proof of Theorem~\ref{thm:balanced_separator_dual_rounding}.

% Since the probability that the sampled subgraph satisfies both the properties stated in Claim~\ref{claim:small_clique} and Claim~\ref{claim:no_balanced_separator} is at least $\frac{1}{2}$, there exists an induced subgraph $H \subseteq G$ that satisfies both. Let $H$ be such an induced subgraph. By Claim~\ref{claim:small_clique} every induced subgraph of $H$ with independence number at most $b$, has at most $\frac{3 \cdot a \cdot \ell}{\lpopt}$ vertices. Furthermore, combining Lemma~\ref{lem:separation_vs_separators} with Claim~\ref{claim:no_balanced_separator}, we conclude that $H$ has no $(I,\frac{1}{2})$-balanced separator $S$ with $\mathrm{cov}_{\mathcal{F}}(S) \leq \lpopt$, thereby completing the proof of Theorem~\ref{thm:balanced_separator_dual_rounding}.
\end{proof}

\section{Proofs of Main Theorems}\label{sec:proofOfTheorems}
    \begin{theorem}\label{thm:mengers_with_F}
    Let $G$ be a graph, $A, B \subseteq V(G)$ be vertex subsets, $f$, $a$ be positive integers, and ${\cal F}$ be a family of vertex subsets satisfying $\alpha(F)\leq a$ for every $F\in\mathcal{F}$. Then, either there exists an $A$--$B$ separator $S$ in $G$ such that $\mathrm{fcov}_{\mathcal{F}}(S) \leq 12 \cdot a \cdot f \cdot \log 2n$, or, for every $\ell \geq \log 2|\mathcal{F}|$, there exists a multiset $\mathcal{Q}$ of induced $A$--$B$ paths in $G$ of cardinality at least $f \cdot \ell$, such that for every $F \in \mathcal{F}$, the number of paths in $\mathcal{Q}$ that have a non-empty intersection with $F$ is at most $6 (\ell+1)$.
\end{theorem}
\begin{proof}
    Let $\mathcal{P}$ represent the set of induced $A$--$B$ paths in $G$. We recall the $A$--$B$ separator linear program corresponding to the instance $(G, A, B, \mathcal{F})$, formulated using non-negative real variables $\{x_F\}_{F \in \mathcal{F}}$.
    \begin{align*}
        \text{Minimize :}&\quad \sum\limits_{F\in\mathcal{F}} x_{F}\\
        \text{subject to :}&\quad \sum\limits_{\substack{F\in \mathcal{F}\\ F\cap P\neq \emptyset}} x_{F} \geq 1 && \forall P\in \mathcal{P}
    \end{align*}
    Let $\{x_F\}_{F \in \mathcal{F}}$ be an optimal solution to this linear program. Observe that $\{x_F\}_{F \in \mathcal{F}}$ is a fractional $(A, B)$-separator. This is because, if the constraint $\sum\limits_{\substack{F\in \mathcal{F}\\ F\cap P\neq \emptyset}} x_{F} \geq 1$ holds for every induced $A$--$B$ path $P$, then it also holds for every $A$--$B$ path, since for any such path $P$ in $G$, there exists an induced $A$--$B$ path contained within $G[P]$. Hence, if $\sum_{F \in \mathcal{F}} x_F \leq f$, then applying Theorem~\ref{thm:A-B_separator} to the instance $(G, A, B, \mathcal{F}, \{x_F\}_{F \in \mathcal{F}})$ guarantees the existence of an $A$--$B$ separator $S$ in $G$ such that
    $$
    \mathrm{fcov}_{\mathcal{F}}(S) \quad
    \leq \quad 12 \cdot a \cdot \log 2n \cdot \sum_{F \in \mathcal{F}} x_F \quad
    \leq \quad 12 \cdot a \cdot f \cdot \log 2n.
    $$
    Otherwise, let $f'$ denote the minimum value of $\sum_{F \in \mathcal{F}} x_F$ over all fractional $(A, B)$-separators $\{x_F\}_{F \in \mathcal{F}}$, and note that $f' > f$. Applying Theorem~\ref{thm:A-B_path_packing} to the instance $(G, A, B, \mathcal{F}, f')$ guarantees the existence of a multiset $\mathcal{Q}$ of induced $A$--$B$ paths in $G$ of cardinality at least $f' \cdot \ell \geq f \cdot \ell$, such that for every $F \in \mathcal{F}$, the number of paths in $\mathcal{Q}$ intersecting $F$ is at most $6 (\ell+1)$. This concludes the proof of the theorem.
\end{proof}

Now we relate the independence number of a set to its fractional cover number using families whose elements have low independence number. This helps us transition between fractional cover number using $(b,a)$-container families and independence number, in Theorems~\ref{thm:ab_separator_in_class} and~\ref{thm:tw_and_talpha}.

\begin{lemma}\label{lem:fractional_cover_independence_bound}
    Let $G$ be a graph, $S\subseteq V(G)$ be a vertex subset, $a$ be a positive integer and $\mathcal{F}$ be a family of vertex subsets of $G$ such that $\alpha(F)\leq a$ for every $F\in\mathcal{F}$. Then $\alpha(S) \leq a \cdot \mathrm{fcov}_{\mathcal{F}}(S)$
\end{lemma}
\begin{proof}
    Let $I$ be an independent set in $G[S]$ such that $\alpha(S) = |I|$. Let $\{x_F\}_{F \in \mathcal{F}}$ be a fractional cover of $S$, satisfying $\sum_{F \in \cal F} x_F = \mathrm{fcov}_{\cal F}(S)$. We have that,
    $$
    % a \cdot \mathrm{fcov}_{\cal F}(I) 
    a \cdot \sum_{F \in \cal F} x_F \quad
    \geq \quad a \cdot \sum_{F \in {\cal F}} \sum_{v \in F \cap I} \frac{x_F}{a} \quad
    = \quad \sum_{v \in I} \sum_{\substack{F \in {\cal F} \\ F \ni v}} x_F \quad
    \geq \quad \sum_{v \in I} 1 \quad
    \geq \quad |I|,
    % = \alpha(S)
    $$
    where the first inequality uses the fact that $\alpha(F) \leq a$ for all $F \in \mathcal{F}$. Hence, we conclude that $\alpha(S) \leq a \cdot \mathrm{fcov}_{\mathcal{F}}(S)$.
\end{proof}

\MainThmSeparator*

% \begin{theorem}\label{thm:ab_separator_in_graph_class}
% For every positive integer $c$ there exists an integer $d=d(c)$ with the following property. If $\mathcal C$ is a hereditary graph class such that for every $G \in \mathcal C$ on at least $3$ vertices and for every two non-adjacent vertices $u,v \in V(G)$, there exists a set $X \subseteq V(G)$ disjoint from $\{u,v\}$ with $|X| \leq (\omega(G) \log |V(G)|)^c$ that separates $u$ from $v$, then for every $G \in \mathcal C$ on at least $3$ vertices and for every two non-adjacent vertices $u,v \in V(G)$, there exists a set $X \subseteq V(G)$ disjoint from $\{u,v\}$, with $\alpha(X) \leq \log^d (|V(G)|)$, that separates $u$ from $v$.
% \end{theorem}

% \begin{theorem}\label{thm:ab_separator_in_graph_class}
%     For every positive integer $c$ there exists an integer $d=d(c)$ with the following property. Let $\mathcal C$ be a hereditary graph class such that for every $G \in \mathcal C$ and for every two non-adjacent vertices $u,v \in V(G)$, there exists a set $X \subseteq V(G)$ disjoint from $\{u,v\}$ with $|X| \leq (\omega(G) \log |V(G)|)^c$ that separates $u$ from $v$. Then for every $G \in \mathcal C$ and for every two non-adjacent vertices $u,v \in V(G)$, there exists a set $X \subseteq V(G)$ disjoint from $\{u,v\}$, with $\alpha(X) \leq \log^d (4|V(G)|)$, that separates $u$ from $v$.
% \end{theorem}

\begin{proof}
    We may assume that $|V(G)| \geq 3$ for otherwise the conclusion of the theorem holds taking $X=\emptyset$.
    Let $t$ be the smallest positive integer such that $t > (2 \log(2t))^c$. Consider the complete bipartite graph $K_{t,t}$ on $2t$ vertices, and let $u, v$ be two vertices on the same side of the bipartition in $K_{t,t}$. Observe that $u$ and $v$ are non-adjacent. Furthermore, any vertex subset $X \subseteq V(K_{t,t}) \setminus \{u, v\}$ that separates $u$ from $v$ has size at least $t$. Finally note that $\omega(K_{t,t}) = 2$.  Therefore, by the definition of $\mathcal{C}$ and the choice of $t$, it follows that $\mathcal{C}$ does not contain $K_{t,t}$. Let $G \in \mathcal{C}$. Since $\mathcal{C}$ is hereditary, $G$ is $\overline{2K_t}$-free. Therefore, by Lemma~\ref{lem:containers1}, $G$ has a $(1,a)$-container family $\mathcal{F}$ of size at most ${(|V(G)|+1)}^{(2t \log |V(G)| + 3)^4}$ for $a \leq (2t \log |V(G)| + 3)^4$.

    Let $u,v\in V(G)$ be an arbitrary pair of non-adjacent vertices. Define $A := N(u)$, $B := N(v)$, $G' := G - \{u, v\}$, and $f := 6 \cdot \lceil (12 \cdot a \cdot {(\lceil \log 2|\mathcal{F}| \rceil + 3)} \cdot \log|V(G)|)^c + 1 \rceil$. Note that since $u$ and $v$ are non-adjacent, we have $A, B \subseteq V(G')$. We apply Theorem~\ref{thm:mengers_with_F} to the tuple $(G', A, B, f, a, \mathcal{F})$ with the parameter $\ell = \lceil \log 2|\mathcal{F}| \rceil$, which leads to two possible cases.

    Suppose there exists a multiset $\mathcal{Q}$ of induced $A$--$B$ paths in $G'$ of cardinality at least $f \cdot \ell$, such that for every $F \in \mathcal{F}$, the number of paths in $\mathcal{Q}$ that have a non-empty intersection with $F$ is at most ${ 6(\ell+1)}$. Then let $H := G[\{u,v\}\cup\bigcup_{P\in\mathcal{Q}}V(P)]$ and note that since $u$ and $v$ are non-adjacent, $H$ has at least three vertices. As $\mathcal{F}$ is a $(1,a)$-container family, every clique in $G$, and consequently in the induced subgraph $H$, is contained in some $F \in \mathcal{F}$. Furthermore, since the paths in $\mathcal{Q}$ are induced and each $F \in \mathcal{F}$ has independence number at most $a$, we have $|V(P) \cap F| \leq 2a$ for every $P \in \mathcal{Q}$ and $F \in \mathcal{F}$. It follows that $\omega(H) \leq \max_{F\in\mathcal{F}}\{|F\cap V(H)|\} \leq 12\cdot a\cdot { (\ell+1) + 2}$.
    As $\mathcal{C}$ is hereditary and $H$ is an induced subgraph of $G$, $H$ belongs to the family $\mathcal{C}$. Also, because $u$ and $v$ are non-adjacent in $H$, there is  a set $X \subseteq V(H)$ disjoint from $\{u,v\}$ with $|X| \leq (12 \cdot a \cdot { (\ell+3)} \cdot \log |V(H)|)^c$ that separates $u$ and $v$ in $H$. But, as $|X|\geq \mathrm{cov}_{\mathcal{F}}(X)$, we have a subfamily $\mathcal{F}'\subseteq\mathcal{F}$ of cardinality at most $(12 \cdot a \cdot { (\ell+3)} \cdot \log |V(H)|)^c < \frac{f}{6}$ such that $\bigcup_{F\in\mathcal{F}'}F$ separates $u$ from $v$ in $H$. However, this leads to a contradiction as $|\mathcal{Q}| \geq f\cdot\ell$ and as the number of paths in $\mathcal{Q}$ that intersect $F$ is at most $6  { (\ell+1)}$ for every $F \in \mathcal{F}$, which implies that $|\mathcal{F}'|$ is at least $\frac{f}{6}$.
    
    % But since, for every $F \in \mathcal{F}$, the number of paths in $\mathcal{Q}$ that intersect $F$ is at most $6 \ell$, and since $|\mathcal{Q}| \geq f \cdot \ell$, we can find a subfamily $\mathcal{Q}' \subseteq \mathcal{Q}$ such that $|\mathcal{Q}'| \geq \frac{f}{6}$ and, for every $F \in \mathcal{F}$, at most one path in $\mathcal{Q}'$ intersects $F$. Observe that the paths in $\mathcal{Q}'$ are vertex-disjoint, since every vertex belongs to some $F \in \mathcal{F}$ and no two paths in $\mathcal{Q}'$ intersect the same $F$. By Menger's theorem, this implies that any vertex subset $X \subseteq V(H)$ that is disjoint from ${u, v}$ and separates $u$ from $v$ has size at least $\frac{f}{6}$. However, this leads to a contradiction, since $\frac{f}{6} > (12 \cdot a\cdot \ell \cdot \log|V(H)|)^c$.

    Therefore, $G'$  has an $A$--$B$ separator $X$ with $\mathrm{fcov}_{\mathcal{F}}(X) \leq 12 \cdot a \cdot f \cdot \log(2|V(G')|)$. But, $X$ is a set disjoint from $\{u, v\}$ that separates $u$ from $v$ in $G$ and, by Lemma~\ref{lem:fractional_cover_independence_bound}, satisfies $\alpha(X) \leq 12 \cdot a^2 \cdot f \cdot \log(2|V(G')|)$. Hence, we define $d(c)$ to be the smallest positive integer such that
    $$\log^{d(c)} (n) \geq 12 \cdot (2t \log n + 3)^8 \cdot 6 \cdot \lceil (12 \cdot (2t \log n + 3)^4 \cdot { \lceil 4+ (2t \log n + 3)^4\cdot\log (n+1) \rceil} \cdot \log n)^c + 1 \rceil \cdot \log 2n,$$ 
    holds for every positive integer $n$, which completes the proof of the theorem.
\end{proof}

In order to prove Theorem~\ref{thm:tw_and_talpha}, we need the following lemma. We remark that the proof of Lemma~\ref{lem:tree_alpha_separator} follows closely that of Lemma 7.1 in \cite{TI2}.

\begin{lemma}\label{lem:tree_alpha_separator}
    Let $G$ be a graph and let $a$ be a positive integer. If for every independent set $I \subseteq V(G)$ of size at least $a$, there exist disjoint subsets $I_1, I_2 \subseteq I$ and an $I_1$--$I_2$ separator $S$ such that $S \cap (I_1 \cup I_2) = \emptyset$ and $\alpha(S) \leq \min\{|I_1|, |I_2|\}$, then $\atw(G) \leq \frac{3}{2} \cdot a$.
\end{lemma}
\begin{proof}
    We will prove that for every set $Z \subseteq V(G)$ with $\alpha(Z) = a$ there is a tree decomposition $(T,\chi)$ of $G$ such that $\alpha(\chi(t)) \leq \frac{3}{2} \cdot a$ for every $t \in T$, and that there exists $t \in T$ such that $Z \subseteq \chi(t)$. The proof is by induction on $|V(G)|$. Observe that every induced subgraph of $G$ satisfies the assumption of the lemma.
  
    Let $Z \subseteq V(G)$ with $\alpha(Z) = a$. Let $I$ be an independent set of $Z$ with $|I|=\alpha(Z)$. Let $I_1$, $I_2\subseteq I$ be disjoint subsets of $I$ and let $S$ be an $I_1$--$I_2$ separator satisfying $S\cap (I_1\cup I_2) = \emptyset$ and $\alpha(S)\leq \min\{|I_1|, |I_2|\}$. Let $C_1,\dots, C_r$ be the connected components of $G-S$ and define $Z_i := Z\cap V(C_i)$ for every $i\in [r]$. Consider $Z_i$ for some $i\in [r]$ and observe that as $S$ is an $I_1$--$I_2$ separator, $Z_i$ is disjoint and anti-complete from either $I_1$ or $I_2$. Without loss of generality, assume $Z_i$ is disjoint and anti-complete from $I_1$. 
    Let $I'_i$ be an independent set in $Z_i$ such that $|I'_i| = \alpha(Z_i)$. If $|I'_i| > |I|-|I_1|$, then $I'_i\cup I_1$ will be an independent set in $Z$ of size strictly greater than that of $I$, which is a contradiction. Combining this with $|I|\geq |I_1|+|I_2|$, implies that $\alpha(Z_i\cup S) \leq |I|-|I_1| + \min\{|I_1|, |I_2|\} \leq |I| = a$. If $\alpha(C_i) < a$ then let $(T_i,\chi_i)$ be the trivial tree decomposition of $C_i$ with $T_i$ having a single node $t_i$ and $\chi(t_i) := V(C_i)$. 
    Otherwise, let $Z'_i\subseteq V(C_i)\cup S$ be such that $Z_i\cup S\subseteq Z'_i$ and $\alpha(Z'_i) = a$ and let $(T_i,\chi_i)$ be the tree decomposition of $C_i$ obtained by applying our inductive assumption, with $Z'_i\subseteq\chi_i(t_i)$ for some $t_i\in V(T_i)$. Now, let $T$ be the tree obtained from the disjoint union of $T_1, \dots, T_r$ by adding a new vertex $t_0$ which is only adjacent to $t_1, \dots, t_r$. Define $\chi(t)=\chi_i(t)$ for every $t \in T_i$, and let $\chi(t_0) = Z \cup S$. It can be verified that $(T,\chi)$ is a tree decomposition of $G$. Since $\alpha(Z\cup S) \leq a + \min\{|I_1|,|I_2|\} \leq \frac{3}{2} \cdot a$ and since $Z\subseteq \chi(t_0)$, we have that $(T,\chi)$ satisfies the conclusion of the lemma.
\end{proof}

\begin{theorem}\label{thm:tw_and_talpha}
    Let $G$ be a graph, $f$, $a$, and $b$ be positive integers with $a \geq 2$, and let $\mathcal{F}$ be a $(b, a)$-container family in $G$. If $\atw(G) > 1020000 \cdot \lceil \log 2n \cdot a^3 \cdot {\log(a\cdot f + 4)} \cdot f \rceil$, then there exists an induced subgraph $G' \subseteq G$ and an independent set $I \subseteq V(G')$ of size $680000 \cdot \lceil \log 2n \cdot a^3 \cdot {\log(a\cdot f + 4)} \cdot f \rceil$ that satisfies the following properties:
    \begin{itemize}
        \item Every induced subgraph $H$ of $G'$ with $\alpha(H)\leq b$ satisfies
        $$|V(H)| \quad \leq \quad {42} \cdot b \cdot (\lceil\log 4|\mathcal{F}|\rceil + 680000 \cdot \lceil \log 2n \cdot a^3 \cdot {\log(a\cdot f + 4)}\rceil).$$
        % \item Every induced subgraph $H$ of $G'$ with $\alpha(H)\leq b$, satisfies\todo{note to self - make changes here} $$|V(H)| \quad \leq \quad 21 \cdot a \cdot (\lceil\log 4|\mathcal{F}|\rceil + 680000 \cdot \lceil \log 2n \cdot a^3 \cdot {\lceil \log(a\cdot f)\rceil}\rceil)$$
        \item For every $(I,\frac{1}{2})$-balanced separator $S$ in $H$, we have that $\mathrm{cov}_{\mathcal{F}}(S) \geq f$.
    \end{itemize}
\end{theorem}

\begin{proof}
    We begin with the following claim.
    \begin{claim}\label{claim:balanced_sep_of_I}
        There exists an independent set $I \subseteq V(G)$ in $G$, of size $680000 \cdot \lceil \log 2n \cdot a^3 \cdot {\log(a\cdot f + 4)} \cdot f \rceil$ for which no $(I, \frac{95}{100})$-balanced separator $S$ in $G$ satisfies $\mathrm{fcov}_{\mathcal{F}}(S) \leq 17000 \cdot \log 2n \cdot a^2 \cdot {\log(a\cdot f + 4)} \cdot f$.
    \end{claim}
    \begin{proof}
        Assume for contradiction that, for every independent set $I \subseteq V(G)$ such that $|I| = 680000 \cdot \lceil \log 2n \cdot a^3 \cdot {\log(a\cdot f + 4)} \cdot f \rceil$, there exists an $(I, \frac{95}{100})$-balanced separator $S$ in $G$ with $\mathrm{fcov}_{\mathcal{F}}(S) \leq 17000 \cdot \log 2n \cdot a^2 \cdot {\log(a\cdot f + 4)} \cdot f$. Fix such an independent set $I$ and the corresponding separator $S \subseteq V(G)$. Observe that, by Lemma~\ref{lem:fractional_cover_independence_bound}, we have $\alpha(S)\ \leq\ a \cdot \mathrm{fcov}_{\mathcal{F}}(S)\ \leq\ 17000 \cdot \log 2n \cdot a^3 \cdot {\log(a\cdot f + 4)} \cdot\ f \leq\ \frac{1}{40}|I|$. We will show that there exist disjoint subsets $I_1, I_2 \subseteq I$ such that $S$ is an $I_1$--$I_2$ separator with $S \cap (I_1 \cup I_2) = \emptyset$ and $\min\{|I_1|, |I_2|\} \geq \frac{1}{40} |I|$. This suffices to prove the claim via Lemma~\ref{lem:tree_alpha_separator}, since $\atw(G) > 1020000 \cdot \lceil \log 2n \cdot a^3 \cdot {\log(a\cdot f + 4)} \cdot f \rceil = \frac{3}{2}|I|$.
        
        To this end, let $C_1,\dots, C_r$ be components of $G-S$ that intersect $I$ and define $I'_i := C_i\cap I$ for every $i\in [r]$. By relabeling the sets if necessary, assume that $|I'_i| \geq |I'_{i+1}|$ for all $i \in [r]$. If $|I'_1| \geq \frac{|I|}{2}$, then define $I_1 := I'_1$ and $I_2 = I\setminus (I_1\cup S)$. Here, we have $|I_2| = |I| - |I_1| - |S\cap I| \geq |I| - \frac{95}{100}|I| - \frac{1}{40}|I| = \frac{1}{40}|I|$. Otherwise, let $q$ be the smallest integer in $[r]$ such that $\sum_{i=1}^{q}|I'_i| \geq \frac{|I|}{3}$. We claim that $\sum_{i=1}^{q}|I'_i| \leq \frac{2|I|}{3}$. If $q = 1$, the claim holds trivially. Else, observe that
        $$
        \sum_{i=1}^{q}|I'_i| \quad
        = \quad |I'_{q}|\ +\ \sum_{i=1}^{q-1}|I'_i| \quad
        \leq \quad |I'_{q-1}|\ +\ \sum_{i=1}^{q-1}|I'_i| \quad
        \leq \quad 2\sum_{i=1}^{q-1}|I'_i| \quad
        \leq \quad \frac{2|I|}{3}.
        $$
        Define $I_1 := \bigcup_{i\in[q]} I'_i$ and $I_2 := I\setminus (I_1\cup S)$. Since the sets in $\{I'_i\}_{i\in [r]}$ are pairwise disjoint, we have that $|I_1| = \sum_{i=1}^{q}|I'_i|$. Hence, $|I_1| \geq \frac{|I|}{3}$ and since $|I_1| \leq \frac{2|I|}{3}$, we have $|I_2| \geq |I| - |I_1| - |S\cap I| \geq |I| - \frac{2}{3}|I| - \frac{1}{40}|I| \geq \frac{1}{40}|I|$. Hence, in either case, we obtain disjoint subsets $I_1$, $I_2$ of $I$ such that $S$ is an $I_1$--$I_2$ separator and $S \cap (I_1 \cup I_2) = \emptyset$ and $\min\{ |I_1|, |I_2|\} \geq \frac{1}{40}|I|$. This concludes the proof of the claim.
    \end{proof}
    
    Let $I$ be an independent set in $G$ whose existence is guaranteed by Claim~\ref{claim:balanced_sep_of_I}. Namely, $I$ has size $680000 \cdot \lceil \log 2n \cdot a^3 \cdot {\log(a\cdot f + 4)} \cdot f \rceil$, and every $(I, \frac{95}{100})$-balanced separator $S$ in $G$ satisfies $\mathrm{fcov}_{\mathcal{F}}(S) > 17000 \cdot \log 2n \cdot a^2 \cdot {\log(a\cdot f + 4)} \cdot f$. {Observe that the balanced separator linear program for the instance $(G, I, \mathcal{F})$, is feasible and has a bounded solution since $(G, I, \mathcal{F})$ meets the conditions of Observation~\ref{obs:lp_is_feasible}. Indeed since for every vertex $v\in V(G)$, $\{v\}$ induces a subgraph of independence number $1$, which implies that there exists a set in $\mathcal{F}$ that contains $v$. Furthermore the objective value of any optimal solution must exceed $f$ since otherwise,} by Theorem~\ref{thm:balanced_separator_rounding}, the set $I$ would admit a balanced separator $S$ with $\mathrm{fcov}_{\mathcal{F}}(S) \leq 17000 \cdot \log 2n \cdot a^2 \cdot {\log(a\cdot f + 4)} \cdot f$, contradicting the definition of $I$. 

    By strong duality, the dual linear program described in Section~\ref{sec:dual_of_balanced_separator} also has an optimal objective value, denoted by $f'$, that is at least $f$. Now, applying Theorem~\ref{thm:balanced_separator_dual_rounding} to the instance $(G, I, \mathcal{F})$, with $\ell = 7 \cdot (\lceil f' \cdot \log 4|\mathcal{F}|\rceil + 680000 \cdot \lceil \log 2n \cdot a^3 \cdot {\log(a\cdot f + 4)} \cdot f \rceil)$ guarantees the existence of a subgraph $G' \subseteq G$ with $I \subseteq V(G')$ such that the following holds: every induced subgraph $H$ of $G'$ with $\alpha(H) \leq b$ satisfies
    \begin{align*}
    |V(H)| \quad
    &\leq\quad 3 \cdot b \cdot \frac{\ell}{f'} \\
    &\leq\quad {21} \cdot b \cdot \frac{1}{f'}\cdot (\lceil f' \rceil\cdot \lceil \log 4|\mathcal{F}| \rceil + 680000 \cdot \lceil \log 2n \cdot a^3 \cdot {\log(a\cdot f + 4)} \rceil\cdot \lceil f \rceil)\\
    &\leq\quad {42} \cdot b \cdot (\lceil \log 4|\mathcal{F}| \rceil + 680000 \cdot \lceil \log 2n \cdot a^3 \cdot {\log(a\cdot f + 4)} \rceil),
    \end{align*}
    and for every $(I, \frac{1}{2})$-balanced separator $S$ in $G'$, we have $\mathrm{cov}_{\mathcal{F}}(S) \geq f$. This concludes the proof of the theorem.
    
    % By strong duality, the dual LP described in Section~\ref{sec:dual_of_balanced_separator}  also has an optimal objective value, denoted by $f'$, that is at least $f$. Now, applying Theorem~\ref{thm:balanced_separator_dual_rounding} to the instance $(G, I, \mathcal{F})$, with $\ell = 7 \cdot (\lceil f' \cdot \log 4|\mathcal{F}|\rceil + 680000 \cdot \lceil \log 2n \cdot a^3 \cdot {\color{blue} \lceil \log(f \cdot a)\rceil} \cdot f \rceil)$, guarantees the existence of a subgraph $G' \subseteq G$ with $I \subseteq V(G')$ such that the following holds: every induced subgraph $H$ of $G'$ with $\alpha(H) \leq b$ satisfies $|V(H)| \leq 3 \cdot a \cdot \frac{\ell}{f'} \leq 21 \cdot a \cdot (\lceil \log 4|\mathcal{F}| \rceil + 680000 \cdot \lceil \log 2n \cdot a^3 \cdot {\color{blue} \lceil \log(f \cdot a)\rceil} \rceil)$, and for every $(I, \frac{1}{2})$-balanced separator $S$ in $G'$, we have $\mathrm{cov}_{\mathcal{F}}(S) \geq f$. This concludes the proof of the theorem.
\end{proof}

For the proof of Theorem~\ref{thm:main}, we also need the following propositions. 

\begin{proposition}[\cite{erdos1989ramsey}]\label{prop:erdos_hajnal}
    For every positive integer $t$ there exist positive real numbers $c(t)$ and $\varepsilon(t)$ such that every $K_{t,t}$-free graph $G$ on $n$ vertices satisfies $\omega(G) \geq c(t)\cdot n^{\varepsilon(t)}$ or $\alpha(G) \geq c(t)\cdot n^{\varepsilon(t)}$.
\end{proposition}
\begin{proposition}[\cite{cygan2015parameterized}]\label{prop:tree_width_separator}
    Let $G$ be a graph, and let $k$ be a positive integer. If $\tw(G)\leq k$, then for every set $Z \subseteq V(G)$, there exists a $(Z,\frac{1}{2})$-balanced separator $S$ such that $|S|\leq k$.
\end{proposition}

We remark that Proposition~\ref{prop:erdos_hajnal} is not explicitly stated as such by Erd\H{o}s and Hajnal~\cite{erdos1989ramsey}, but is an immediate implication of Theorem 1.2 in~\cite{erdos1989ramsey}.

\MainThmTreeAlpha*

% \begin{theorem}\label{thm:polylog_equivalances}
% Let $\mathcal{C}$ be a hereditary graph class. The following are equivalent:
% \begin{enumerate}\setlength\itemsep{-.7pt}
%     \item[(i)]\label{itm:atwbounded} $\exists c_1 > 0$ such that for every $G \in \mathcal{C}$ on at least $3$ vertices we have $\atw(G) \leq (\log |V(G)|)^{c_1}$
%     \item[(ii)]\label{itm:atwcliquebounded} $\exists c_2 > 0$ such that for every $G \in \mathcal{C}$ on at least $3$ vertices we have $\atw(G) \leq (\omega(G) \log |V(G)|)^{c_2}$
%     \item[(iii)]\label{itm:twcliquebounded} $\exists c_3 > 0$ such that for every $G \in \mathcal{C}$ on at least 3 vertices we have $\tw(G) \leq (\omega(G) \log |V(G)|)^{c_3}$
% \end{enumerate}
% \end{theorem}

% \begin{theorem}\label{thm:polylog_equivalances}
% Let $\mathcal{C}$ be a hereditary graph class. The following are equivalent:
% \begin{itemize}
%     \item[(i)] There exists a constant $c_1$ such that, for every $G \in \mathcal{C}$, we have $\atw(G) \leq (\log\, 4|V(G)|)^{c_1}$
%     \item[(ii)] There exists a constant $c_2$ such that, for every $G \in \mathcal{C}$, we have $\atw(G) \leq (\omega(G) \log\, 4|V(G)|)^{c_2}$
%     \item[(iii)] There exists a constant $c_3$ such that, for every $G \in \mathcal{C}$, we have $\tw(G) \leq (\omega(G) \log\, 4|V(G)|)^{c_3}$
% \end{itemize}
% \end{theorem}

\begin{proof}

    {\em (i) $\implies$ (ii)} \quad Trivially true.\\[-8pt]
    
    {\em (ii) $\implies$ (iii)} \quad Let $t$ be the smallest integer such that $t > (2 \log(2t))^{c_2}$. Consider the complete bipartite graph $K_{t,t}$ on $2t$ vertices. Since any tree decomposition $(T,\chi)$ of $K_{t,t}$ has a node $t\in V(T)$ such that $\chi(t)$ contains the closed neighborhood of some vertex, we have $\tw(G)\geq t$. Furthermore, $\omega(K_{t,t}) = 2$. Therefore, by the definition of $\mathcal{C}$ and the choice of $t$, it follows that $\mathcal{C}$ does not contain $K_{t,t}$. 
    
    Let $G \in \mathcal{C}$. Since $\mathcal{C}$ is hereditary, $G$ is  $K_{t,t}$-free (equivalently $\overline{2K_t}$-free). Let $(T, \chi)$ be a tree decomposition of $G$ such that $\alpha(\chi(v)) \leq \atw(G)$ for all $v \in V(T)$. Applying Proposition~\ref{prop:erdos_hajnal} to $G[\chi(v)]$, where $v\in V(T)$, implies that $|\chi(v)| < \frac{1}{c(t)}(\omega(G)\cdot \atw(G))^{\frac{1}{\varepsilon(t)}}$ for some positive constants $c(t)$ and $\varepsilon(t)$. Therefore, defining $c_3$ to be the smallest positive integer that satisfies 
    $$
    (\omega(G) \log\, |V(G)|)^{c_3} \quad 
    \geq \quad 
    \frac{1}{c(t)}(\omega(G)\cdot (\omega(G) \log\, |V(G)|)^{c_2})^{\frac{1}{\varepsilon(t)}},
    $$ 
    suffices to prove the implication.\\[-8pt]
    
    {\em (iii) $\implies$ (i)} \quad Let $t$ be the smallest integer such that $t > (2 \log(2t))^{c_3}$. Consider the complete bipartite graph $K_{t,t}$ on $2t$ vertices. Since any tree decomposition $(T,\chi)$ of $K_{t,t}$ has a node $t\in V(T)$ such that $\chi(t)$ contains the closed neighborhood of some vertex, we have $\tw(G)\geq t$. Furthermore, $\omega(K_{t,t}) = 2$. Therefore, by the definition of $\mathcal{C}$ and the choice of $t$, it follows that $\mathcal{C}$ does not contain $K_{t,t}$. 
    
    Let $G \in \mathcal{C}$. Since $\mathcal{C}$ is hereditary, $G$ is  $K_{t,t}$-free (equivalently $\overline{2K_t}$-free). Therefore, by Lemma~\ref{lem:containers1}, $G$ has a $(1,a)$-container family $\mathcal{F}$ of size at most $({|V(G)|+1})^{(2t \log |V(G)| + 3)^4}$ for $a \leq (2t \log |V(G)| + 3)^4$. Choose $f$ to be the smallest positive integer satisfying 
    $$
    f \quad > \quad ({ 42} \cdot (\lceil\log 4|\mathcal{F}|\rceil + 680000 \cdot \lceil \log 2|V(G)| \cdot a^3 \cdot { \log(a\cdot f + 4)}\rceil) \cdot \log |V(G)|)^{c_3}.
    $$
    % $$
    % f \quad > \quad (21 \cdot a \cdot (\lceil\log 4|\mathcal{F}|\rceil + 680000 \cdot \lceil \log 2|V(G)| \cdot a^3 \cdot {\color{blue} \lceil \log(f \cdot a)\rceil}\rceil)) \cdot \log |V(G)|)^{c_3}
    % $$ 
    %
    We apply Theorem~\ref{thm:tw_and_talpha} to $G$, $f$, and $\mathcal{F}$, and observe that, if $\atw(G) > 1020000 \cdot \lceil \log 2|V(G)| \cdot a^3 \cdot { \log(a\cdot f + 4)} \cdot f \rceil$, then there exists an induced subgraph $G' \subseteq G$ with the following properties: $\omega(G')\, \leq\, { 42} \cdot (\lceil\log 4|\mathcal{F}|\rceil + 680000 \cdot \lceil \log 2|V(G)| \cdot a^3 \cdot { \log(a\cdot f + 4)}\rceil)$; and there exists a set $I \subseteq V(G')$ such that every $(I,\frac{1}{2})$-balanced separator $S$ in $G'$ satisfies $|S| \geq \mathrm{cov}_{\mathcal{F}}(S) \geq f$. By Proposition~\ref{prop:tree_width_separator}, we conclude that $\tw(G') \geq f$. Also, since $|V(G')|$ is at least $f \geq 3$, and since $\mathcal{C}$ is hereditary we have $G'\in\mathcal{C}$. This leads to a contradiction due to our choice of $f$, as $G'$ now satisfies $\tw(G') \leq (\omega(G')\log\, |V(G')|)^{c_3} < f$. Thus we conclude that $\atw(G) \leq 1020000 \cdot \lceil \log 2|V(G)| \cdot a^3 \cdot { \log(a\cdot f + 4)} \cdot f \rceil$.

    Therefore, for every $n\in \mathbb{N}$, if we let $a_n := (2t \log n + 3)^4$ and let $f_n$ denote the smallest positive integer such that,
    $$ 
    f_n \quad
    > \quad (({ 42} \cdot 
    (\lceil a_n \cdot{\log (n+1)}+ 2\rceil + 680000 \cdot \lceil \log 2n \cdot a_n^3 \cdot { \log(a_n \cdot f_n + 4)}\rceil)) \cdot
    \log n)^{c_3},
    $$
    % $$ 
    % f_n \quad
    % > \quad (21 \cdot a_n \cdot 
    % (\lceil a_n \log n + 2\rceil + 680000 \cdot \lceil \log 2n \cdot a_n^3 \cdot \log(f_n \cdot a_n)\rceil)) \cdot
    % \log n)^{c_3}
    % $$
    then, defining $c_1$ to be the smallest positive integer that satisfies 
    $$
    \log^{c_1} n \quad 
    \geq \quad 1020000 \cdot \lceil \log 2n \cdot a_n^3 \cdot { \log(a_n \cdot f_n + 4)} \cdot f_n \rceil,
    $$ 
    for every positive integer greater than 2, suffices to prove the implication and consequently the theorem.
\end{proof}

\section{Conclusion}\label{sec:conclusion}
    We have shown a ``poly-logarithmic'' variant of the recently disproved Dallard-Milani\v{c}-{\v{S}}torgel Conjecture 
(Conjecture~\ref{Milanic}). Our main result is that for every hereditary graph class ${\cal C}$, every graph in ${\cal C}$ has poly-logarithmic tree-independence number if and only if every graph in ${\cal C}$ has treewidth upper bounded by a polynomial in $\log n$ and the size of its maximum clique. 
It remains to characterize the classes ${\cal C}$ such that every graph in ${\cal C}$ has poly-logarithmic tree-independence number. The following conjecture, if true, would be sufficient to characterize all such classes that are closed under {\em induced minors} (that is, closed under vertex deletion and edge contraction). 
\begin{conjecture}
For every positive integer $t$ there exists an integer $c$ such that every graph $G$ on at least $3$ vertices either contains $K_{t,t}$ or $\boxplus_t$ as an induced minor, or satisfies $\atw(G) \leq (\log n)^c$.
\end{conjecture}
Here, for every positive integer $t$ the graph $\boxplus_t$ is the $t \times t$ grid, defined as the graph with vertex set $[t] \times [t]$ where every pair $(i, j)$, $(i', j')$ of vertices are adjacent if and only if $|i-i'| + |j-j'| = 1$.

On the way to showing our main results we defined the notion of independence-containers, and showed that every hereditary graph class excluding the graph $\overline{kK_k}$ for some positive integer $k$ admits container families of quasi-polynomial size that cover all vertex sets with poly-logarithmic independence number. We leave behind the following open problem.

\begin{problem}
    For every pair of integers $k$ and $b$, does there exist an integer $a$ such that every graph $G$ which excludes $\overline{kK_k}$ as an induced subgraph has a $(b, a)$-container family ${\cal F}$ of size at most $n^{a}$?
\end{problem}

\section{Acknowledgement}\label{sec:acknowledgement}
    We would like to thank Julien Codsi and Yori Zwols for pointing out some errors in Section~\ref{sec:ab_seps} of our initial manuscript. We would also like to express our sincere gratitude to the anonymous reviewer whose constructive feedback helped improve this manuscript.

\bibliographystyle{plain}
\bibliography{ref_journal}

@article {deathstar,
    AUTHOR = {Bonamy, Marthe and Bonnet, \'{E}douard and D\'{e}pr\'{e}s,
              Hugues and Esperet, Louis and Geniet, Colin and Hilaire,
              Claire and Thomass\'{e}, St\'{e}phan and Wesolek, Alexandra},
     TITLE = {Sparse graphs with bounded induced cycle packing number have
              logarithmic treewidth},
   JOURNAL = {J. Combin. Theory Ser. B},
  FJOURNAL = {Journal of Combinatorial Theory. Series B},
    VOLUME = {167},
      YEAR = {2024},
     PAGES = {215--249},
      ISSN = {0095-8956,1096-0902},
   MRCLASS = {05C75 (05C70)},
  MRNUMBER = {4723425},
       DOI = {10.1016/j.jctb.2024.03.003},
       URL = {https://doi.org/10.1016/j.jctb.2024.03.003},
}

@article{Agelos,
  author    = {Georgakopoulos, Agelos and Papasoglu, Panos},
  title     = {Graph Minors and Metric Spaces},
  journal   = {Combinatorica},
  year      = {2025},
  volume    = {45},
  number    = {3},
  pages     = {33},
  doi       = {10.1007/s00493-025-00150-6},
  url       = {https://doi.org/10.1007/s00493-025-00150-6},
  issn      = {1439-6912},
  date      = {2025-06-13},
}

@article{doi:10.1137/1025101,
author = {Wong, Richard T.},
title = {Combinatorial Optimization: Algorithms and Complexity ({C}hristos {H}. {P}apadimitriou and {K}enneth {S}teiglitz)},
journal = {SIAM Review},
volume = {25},
number = {3},
pages = {424-425},
year = {1983},
doi = {10.1137/1025101},
URL = {https://doi.org/10.1137/1025101},
eprint = {https://doi.org/10.1137/1025101}
}

@inproceedings{CoarseGeoKwon,
  author       = {Jungho Ahn and
                  J. Pascal Gollin and
                  Tony Huynh and
                  {O-joung} Kwon},
  editor       = {Yossi Azar and
                  Debmalya Panigrahi},
  title        = {{A coarse Erd{\H{o}}s-P{\'{o}}sa theorem}},
  booktitle    = {Proceedings of the 2025 Annual {ACM-SIAM} Symposium on Discrete Algorithms, {SODA} 2025, New Orleans, LA, USA, January 12-15, 2025},
  pages        = {3363--3381},
  publisher    = {{SIAM}},
  year         = {2025},
doi          = {10.1137/1.9781611978322.109},
}

@article{Ramsey,
    AUTHOR = {Ramsey, Frank P.},
     TITLE = {On a {P}roblem of {F}ormal {L}ogic},
   JOURNAL = {Proc. London Math. Soc. (2)},
  FJOURNAL = {Proceedings of the London Mathematical Society. Second Series},
    VOLUME = {30},
      YEAR = {1929},
    NUMBER = {4},
     PAGES = {264--286}
}

@book{mitzenmacher2017probability,
  title={Probability and computing: Randomization and probabilistic techniques in algorithms and data analysis},
  author={Mitzenmacher, Michael and Upfal, Eli},
  year={2017},
  publisher={Cambridge University Press}
}

@article{hagerup1990guided,
  title={A guided tour of Chernoff bounds},
  author={Hagerup, Torben and R{\"u}b, Christine},
  journal={Information processing letters},
  volume={33},
  number={6},
  pages={305--308},
  year={1990},
  publisher={Elsevier}
}

@article{dms2,
  author       = {Cl{\'{e}}ment Dallard and
                  Martin Milani{\v{c}} and
                  Kenny {\v{S}}torgel},
  title        = {Treewidth versus clique number. {II.} Tree-independence number},
  journal      =  {Journal of Combinatorial Theory, Series B},
  volume       = {164},
  pages        = {404--442},
  year         = {2024}
}

@article{dms3,
    author       = {Cl{\'{e}}ment Dallard and
                  Martin Milani{\v{c}} and
                  Kenny {\v{S}}torgel},
title = {Treewidth versus clique number. {III}. Tree-independence number of graphs
 with a forbidden structure},
journal = {Journal of Combinatorial Theory, Series B},
volume = {167},
pages = {338-391},
year = {2024},
}

@inproceedings{lima2024tree,
  author       = {Paloma T. Lima and
                  Martin Milani{\v{c}} and
                  Peter Mur{\v{s}}i{\v{c}} and
                  Karolina Okrasa and
                  Pawel Rz{\k{a}}{\.{z}}{\'{e}}wski and
                  Kenny {\v{S}}torgel},
  editor       = {Timothy M. Chan and
                  Johannes Fischer and
                  John Iacono and
                  Grzegorz Herman},
  title        = {Tree Decompositions Meet Induced Matchings: Beyond Max Weight Independent
                  Set},
  booktitle    = {32nd Annual European Symposium on Algorithms, {ESA} 2024, September
                  2-4, 2024, Royal Holloway, London, United Kingdom},
  series       = {LIPIcs},
  volume       = {308},
  pages        = {85:1--85:17},
  publisher    = {Schloss Dagstuhl - Leibniz-Zentrum f{\"{u}}r Informatik},
  year         = {2024}
}

@article{RobertsonS95b,
  author       = {Neil Robertson and
                  Paul D. Seymour},
  title        = {Graph Minors. {XIII}. {T}he Disjoint Paths Problem},
  journal      = {J. Comb. Theory {B}},
  volume       = {63},
  number       = {1},
  pages        = {65--110},
  year         = {1995}
}

@article{CT,
     author = {Chudnovsky, Maria and Trotignon, Nicolas},
     title = {On treewidth and maximum cliques},
     journal = {Innovations in Graph Theory},
     pages = {223--243},
     year = {2025},
     publisher = {Stichting Innovations in Graph Theory},
     volume = {2},
     doi = {10.5802/igt.11},
     language = {en},
     url = {https://igt.centre-mersenne.org/articles/10.5802/igt.11/}
}

@article{TI2,
title = {Tree independence number {II}. Three-path-configurations},
journal = {Journal of Combinatorial Theory, Series B},
volume = {176},
pages = {74-96},
year = {2026},
issn = {0095-8956},
doi = {https://doi.org/10.1016/j.jctb.2025.08.003},
url = {https://www.sciencedirect.com/science/article/pii/S0095895625000590},
author = {Maria Chudnovsky and Sepehr Hajebi and Daniel Lokshtanov and Sophie Spirkl}
}

@article {awesomegraphparameters,
    AUTHOR = {Kenny Bešter Štorgel and Clément Dallard and Vadim Lozin and Martin Milanič and Viktor Zamaraev},
     TITLE = {{A}wesome graph parameters},
   JOURNAL = {{\rm Preprint available at \url{https://arxiv.org/abs/2511.05285}}},
  FJOURNAL = {},
    VOLUME = {},
      YEAR = {},
     PAGES = {},
      ISSN = {},
   MRCLASS = {},
  MRNUMBER = {},
MRREVIEWER = {},
       DOI = {},
       URL = {},
}

@article {sparseinduced,
    AUTHOR = {Kenny Bešter Štorgel and Clément Dallard and Vadim Lozin and Martin Milanič and Viktor Zamaraev},
     TITLE = {{F}inding large sparse induced subgraphs in graphs of small (but not very small) tree-independence number},
   JOURNAL = {{\rm Preprint available at \url{https://arxiv.org/abs/2601.15861}}},
  FJOURNAL = {},
    VOLUME = {},
      YEAR = {},
     PAGES = {},
      ISSN = {},
   MRCLASS = {},
  MRNUMBER = {},
MRREVIEWER = {},
       DOI = {},
       URL = {},
}

@inproceedings{Yolov,
  author       = {Nikola Yolov},
  editor       = {Artur Czumaj},
  title        = {Minor-matching hypertree width},
  booktitle    = {Proceedings of the Twenty-Ninth Annual {ACM-SIAM} Symposium on Discrete
                  Algorithms, {SODA} 2018, New Orleans, LA, USA, January 7-10, 2018},
  pages        = {219--233},
  publisher    = {{SIAM}},
  year         = {2018}
}

@article {RS-GMXVI,
    AUTHOR = {Robertson, Neil and Seymour, P. D.},
     TITLE = {Graph minors. {XVI}. {E}xcluding a non-planar graph},
   JOURNAL = {J. Combin. Theory Ser. B},
  FJOURNAL = {Journal of Combinatorial Theory. Series B},
    VOLUME = {89},
      YEAR = {2003},
    NUMBER = {1},
     PAGES = {43--76},
      ISSN = {0095-8956,1096-0902},
   MRCLASS = {05C83 (05C10 05C75)},
  MRNUMBER = {1999736},
MRREVIEWER = {D.\ S.\ Archdeacon},
       DOI = {10.1016/S0095-8956(03)00042-X},
       URL = {https://doi.org/10.1016/S0095-8956(03)00042-X},
}

@article {mainconj,
    AUTHOR = {Sintiari, Ni Luh Dewi and Trotignon, Nicolas},
     TITLE = {({T}heta, triangle)-free and (even hole, {$K_4$})-free
              graphs---part 1: {L}ayered wheels},
   JOURNAL = {J. Graph Theory},
  FJOURNAL = {Journal of Graph Theory},
    VOLUME = {97},
      YEAR = {2021},
    NUMBER = {4},
     PAGES = {475--509},
      ISSN = {0364-9024,1097-0118},
   MRCLASS = {05C75},
  MRNUMBER = {4313193},
MRREVIEWER = {Zden\v{e}k\ Ryj\'{a}\v{c}ek},
       DOI = {10.1002/jgt.22666},
       URL = {https://doi.org/10.1002/jgt.22666},
}

@article{BalasY89,
  author       = {Egon Balas and
                  Chang Sung Yu},
  title        = {On graphs with polynomially solvable maximum-weight clique problem},
  journal      = {Networks},
  volume       = {19},
  number       = {2},
  pages        = {247--253},
  year         = {1989}
}

@book{cygan2015parameterized,
  title={Parameterized algorithms},
  author={Cygan, Marek and Fomin, Fedor V and Kowalik, {\L}ukasz and Lokshtanov, Daniel and Marx, D{\'a}niel and Pilipczuk, Marcin and Pilipczuk, Micha{\l} and Saurabh, Saket},
  volume={5},
  number={4},
  year={2015},
  publisher={Springer}
}

@phdthesis{korhonen2024computing,
  title={Computing Width Parameters of Graphs},
  author={Korhonen, Tuukka},
  year={2024},
  school={The University of Bergen}
}

@article{bodlaender1998partial,
  title={A partial {$k$}-arboretum of graphs with bounded treewidth},
  author={Bodlaender, Hans L.},
  journal={Theoretical computer science},
  volume={209},
  number={1-2},
  pages={1--45},
  year={1998},
  publisher={Elsevier}
}

@inproceedings{ChudnovskyGHLS25,
  author       = {Maria Chudnovsky and
                  Peter Gartland and
                  Sepehr Hajebi and
                  Daniel Lokshtanov and
                  Sophie Spirkl},
  editor       = {Yossi Azar and
                  Debmalya Panigrahi},
  title        = {Tree Independence Number {IV.} {E}ven-hole-free graphs},
  booktitle    = {Proceedings of the 2025 Annual {ACM-SIAM} Symposium on Discrete Algorithms,
                  {SODA} 2025, New Orleans, LA, USA, January 12-15, 2025},
  pages        = {4444--4461},
  publisher    = {{SIAM}},
  year         = {2025}
}

@article{erdos1989ramsey,
  title={Ramsey-type theorems},
  author={Erd{\"o}s, Paul and Hajnal, Andr{\'a}s},
  journal={Discrete Applied Mathematics},
  volume={25},
  number={1-2},
  pages={37--52},
  year={1989},
  publisher={Elsevier}
}

@article{beyondTreewidthSurvey,
    author = {Hlinen{\'{y}}, Petr and Oum, Sang-il and Seese, Detlef and Gottlob, Georg},
    title = {Width Parameters Beyond Tree-width and their Applications},
    journal = {The Computer Journal},
    volume = {51},
    number = {3},
    pages = {326-362},
    year = {2007}
}

@inproceedings{KorchemnaL0S024,
  author       = {Viktoriia Korchemna and
                  Daniel Lokshtanov and
                  Saket Saurabh and
                  Vaishali Surianarayanan and
                  Jie Xue},
  title        = {Efficient Approximation of Fractional Hypertree Width},
  booktitle    = {65th {IEEE} Annual Symposium on Foundations of Computer Science, {FOCS}
                  2024, Chicago, IL, USA, October 27-30, 2024},
  pages        = {754--779},
  publisher    = {{IEEE}},
  year         = {2024}
}

@article{saxton2015hypergraph,
  title={Hypergraph containers},
  author={Saxton, David and Thomason, Andrew},
  journal={Inventiones mathematicae},
  volume={201},
  number={3},
  pages={925--992},
  year={2015},
  publisher={Springer}
}

@inproceedings{balogh2018method,
  title={The method of hypergraph containers},
  author={Balogh, J{\'o}zsef and Morris, Robert and Samotij, Wojciech},
  booktitle={Proceedings of the International Congress of Mathematicians: Rio de Janeiro 2018},
  pages={3059--3092},
  year={2018},
  organization={World Scientific}
}

\end{document}